\documentclass[reqno,10pt,centertags]{amsart} 
\usepackage{amsmath,amsthm,amscd,amssymb,latexsym,esint,upref,stmaryrd,
enumerate,verbatim,yfonts,bm,mathtools,mathrsfs}
\usepackage{color}
\usepackage{float}

\usepackage{hyperref} 
\newcommand*{\mailto}[1]{\href{mailto:#1}{\nolinkurl{#1}}}
\newcommand{\arxiv}[1]{\href{http://arxiv.org/abs/#1}{arXiv:#1}}

\usepackage{graphicx}

\newcommand{\R}{{\bbR}}
\newcommand{\N}{{\mathbb N}}

\newcommand{\bbC}{{\mathbb{C}}}

\newcommand{\bbN}{{\mathbb{N}}}

\newcommand{\bbR}{{\mathbb{R}}}
\newcommand{\bbS}{{\mathbb{S}}}

\newcommand{\cB}{{\mathcal B}}

\newcommand{\cH}{{\mathcal H}}

\newcommand{\cL}{{\mathcal L}}

\newcommand{\cX}{{\mathcal X}}
\newcommand{\cY}{{\mathcal Y}}

\newcommand{\p}{{\partial}}

\newcommand{\beq}{\begin{equation}}
\newcommand{\enq}{\end{equation}}

\DeclareMathOperator{\supp}{supp}

\DeclareMathOperator{\dom}{dom}

\DeclareMathOperator*{\slim}{s-lim}

\DeclareMathOperator*{\sgn}{sgn}

\newcommand{\dott}{\,\cdot\,}

\renewcommand{\ln}{\text{\rm ln}}

\newcommand{\no}{\notag}
\newcommand{\lb}{\label}
\newcommand{\f}{\frac}
\newcommand{\df}{\dfrac}
\newcommand{\ul}{\underline}
\newcommand{\ol}{\overline}

\newcommand{\wti}{\widetilde}
\newcommand{\Oh}{O}
\newcommand{\oh}{o}
 
\newcommand{\bi}{\bibitem}

\renewcommand{\ge}{\geqslant}
\renewcommand{\le}{\leqslant}

\renewcommand{\dot}{\overset{\textbf{\Large.}}}

\let\geq\geqslant
\let\leq\leqslant

\makeatletter
\def\theequation{\@arabic\c@equation}

\allowdisplaybreaks 
\numberwithin{equation}{section}

\newtheorem{theorem}{Theorem}[section]
\newtheorem{proposition}[theorem]{Proposition}
\newtheorem{lemma}[theorem]{Lemma}

\newtheorem{definition}[theorem]{Definition}
\newtheorem{hypothesis}[theorem]{Hypothesis}

\theoremstyle{remark}
\newtheorem{remark}[theorem]{Remark}

\begin{document}
\title[On Critical Dipoles in Dimensions $n\geq 3$]{On Critical Dipoles in Dimensions $n\geq 3$}

\author[S.\ B.\ Allan]{S. Blake Allan}
\address{Department of Mathematics, 
Baylor University, Sid Richardson Bldg., 1410 S.\,4th Street,
Waco, TX 76706, USA}
\email{\mailto{Blake\_Allan1@baylor.edu}}

\author[F.\ Gesztesy]{Fritz Gesztesy}
\address{Department of Mathematics, 
Baylor University, Sid Richardson Bldg., 1410 S.\,4th Street,
Waco, TX 76706, USA}
\email{\mailto{Fritz\_Gesztesy@baylor.edu}}
\urladdr{\url{http://www.baylor.edu/math/index.php?id=935340}}

\dedicatory{} 
\date{\today}
\thanks{{\it J. Diff. Eq.} (to appear).}

\begin{abstract}
We reconsider generalizations of Hardy's inequality corresponding to the case of (point) dipole potentials 
$V_{\gamma}(x) = \gamma (u, x) |x|^{-3}$, $x \in \bbR^n \backslash \{0\}$, $\gamma \in [0,\infty)$, 
$u \in \bbR^n$, $|u|=1$, $n \in \bbN$, $n \geq 3$. More precisely, for $n \geq 3$, we provide an alternative proof of the existence of a critical dipole coupling constant $\gamma_{c,n} > 0$, such that 
\begin{align*} 
&\text{for all $\gamma \in [0,\gamma_{c,n}]$, and all $u \in \bbR^n$, $|u|=1$,} \\
&\quad \int_{\bbR^n} d^n x \, |(\nabla f)(x)|^2 \geq \pm \gamma \int_{\bbR^n} d^n x \, (u, x) |x|^{-3} |f(x)|^2, 
\quad f \in D^1(\bbR^n).      
\end{align*}  
with $D^1(\bbR^n)$ denoting the completion of $C_0^{\infty}(\bbR^n)$ with respect to the norm induced by the gradient. Here $\gamma_{c,n}$ is sharp, that is, the largest possible such constant. Moreover, we discuss upper and lower bounds for $\gamma_{c,n} > 0$ and develop a numerical scheme for approximating $\gamma_{c,n}$. 

This quadratic form inequality will be a consequence of the fact    
\[ 
\overline{\big[- \Delta + \gamma (u, x) |x|^{-3}\big]\big|_{C_0^{\infty}(\bbR^n \backslash \{0\})}} \geq 0 
\, \text{ if and only if } \, 0 \leq \gamma \leq \gamma_{c,n} 
\]
in $L^2(\bbR^n)$ (with $\overline{T}$ the operator closure of the linear operator $T$).  

We also consider the case of multicenter dipole interactions with dipoles centered on an infinite discrete set. 
\end{abstract}

\maketitle

{\scriptsize{\tableofcontents}}
\normalsize

\section{Introduction} \lb{s1}

The celebrated (multi-dimensional) Hardy inequality,
\begin{align}
\begin{split} 
\int_{\bbR^n} d^n x \, |(\nabla f)(x)|^2 \geq [(n-2)/2]^2 \int_{\bbR^n} d^n x \, |x|^{-2} |f(x)|^2,& \\ 
f \in D^1(\bbR^n), \; n \in \bbN, \; n \geq 3,&   \lb{1.1}
\end{split} 
\end{align}
the first in an infinite sequence of higher-order Birman--Hardy--Rellich-type inequalities, received enormous attention in the literature due to its ubiquity in self-adjointness and spectral theory problems associated with 
second-order differential operators with strongly singular coefficients, see, for instance, \cite{AGG06},  
\cite[Ch.~1]{BEL15}, \cite[Sect.~1.5]{Da89}, \cite[Ch.~5]{Da95}, \cite{Ge84}, \cite{GP80}, \cite{GM08}, \cite{GM11},  \cite[Part~1]{GM13}, \cite{Ka72}--\cite{KW72}, \cite[Ch.~2, Sect.~21]{OK90}, \cite[Ch.~2]{RS19}, \cite{Si83} and the extensive literature cited therein. We also note that inequality \eqref{1.1} is closely related to Heisenberg's uncertainty relation as discussed in \cite{Fa78}. 

The basics behind the (point) dipole Hamiltonian $- \Delta + V_{\gamma}(x)$, with potential 
\begin{equation}
V_{\gamma}(x) = \gamma \f{(u, x)}{|x|^3}, \quad x \in \bbR^n \backslash \{0\}, \; \gamma \in [0,\infty), 
\; u \in \bbR^n, \; |u| = 1, \; n \geq 3      \lb{1.2} 
\end{equation} 
(with $(a,b)$ denoting the Euclidean scalar product of $a, b \in \bbR^n$), in the physically relevant case $n=3$, have been discussed in great detail in the 1980 paper by Hunziker and G\"unther \cite{HG80}. In particular, these authors point out some of the existing fallacies to be found in the physics literature in connection with dipole potentials and their ability to bind electrons. The primary goal in this paper has been the attempt to extend the three-dimensional results on dipole potentials in \cite{HG80} to the general case $n \geq 4$ and thereby rederiving and complementing some of the results obtained by Felli, Marchini, and Terracini \cite{FMT08}, \cite{FMT09} (see also \cite{FMT07}, \cite{FMO20}, \cite{Te96}). While Felli, Marchini, and Terracini primarily rely on variational techniques, we will focus more on an operator and spectral theoretic approach. To facilitate a comparison between the existing literature on this topic and the results presented in the present paper, we next summarize some of the principal achievements in  \cite{FMT07}, \cite{FMT08}, \cite{FMT09}, \cite{HG80}, \cite{Te96}.

However, we first emphasize that these sources also discuss a number of facts that go beyond the scope of our paper: For instance, Hunziker and G\"unther \cite{HG80} also consider non-binding criteria for Hamiltonians with $M$ point charges and applications to electronic spectra of an $N$-electron Hamiltonian in the presence of $M$ point charges (nuclei). In addition, Felli, Marchini and Terracini \cite{FMT08}, \cite{FMT09} discuss more general operators where the point dipole potential $V_{\gamma}$ in \eqref{1.2} is replaced by\footnote{For simplicity of notation, we will omit the standard surface measure $d^{n-1}\omega$ in $L^{2}(\bbS^{n-1})$, and similarly, the Lebesgue measure $d^nx$ in $L^2(\bbR^n)$.} 
\begin{equation}
a(x/|x|) |x|^{-2}, \; x \in \bbR^n \backslash \{0\}, \quad a \in L^{\infty}(\bbS^{n-1}),     \lb{1.3} 
\end{equation}
and hence \eqref{1.2} represents the special case 
\begin{equation}
a_{\gamma}(x/|x|) = \gamma (u,x/|x|), \quad \gamma \in [0,\infty), \; u \in \bbR^n, \, |u|=1, \; x \in \bbR^n \backslash \{0\}. 
\end{equation}
These authors also provide a discussion of strict positivity of the underlying quadratic form $Q_{\{\gamma_j\}_{1 \leq j \leq M}}(\dott,\dott)$, $M \in \bbN$, in the multi-center case,
\begin{align}
\begin{split}
Q_{\{\gamma_j\}_{1 \leq j \leq M}}(f,f) = \int_{\bbR^n} d^n x \, \bigg[|(\nabla f)(x)|^2 
+ \sum_{j=1}^M \gamma_j \f{(u,(x-x_j))}{|x-x_j|^3} |f(x)|^2\bigg],&     \\
\gamma_j \in [0,\infty), \; x_j \in \bbR^n, \, x_j \neq x_k \text{ for } j \neq k, \, 1 \leq j,k \leq M, \; f \in D^1(\bbR^n),&    
\end{split} 
\end{align}
and its analog with $\gamma_j (u,(x-x_j)/|x-x_j|)$ replaced by $a_{\gamma_j}(\dott)$ restricted to suitable neighborhoods of $x_j$, $1 \leq j \leq M$. In this context also the problem of ``localization of binding'', a notion going back to Ovchinnikov and Sigal \cite{OS79}, is discussed in \cite{FMT09}. In addition, applications to a class of nonlinear PDEs are discussed in 
\cite{Te96}.

Turning to the topics directly treated in this paper and their relation to results in  \cite{FMT07}, \cite{FMT08}, 
\cite{FMT09}, \cite{HG80}, \cite{Te96}, we start by noting that the dipole-modified Hardy-type inequality reads as follows: For each $n \geq 3$, there exists a critical dipole coupling constant $\gamma_{c,n} > 0$, such that 
\begin{align} 
\begin{split} 
&\text{for all $\gamma \in [0,\gamma_{c,n}]$, and all $u \in \bbR^n$, $|u|=1$,} \\
&\quad \int_{\bbR^n} d^n x \, |(\nabla f)(x)|^2 \geq \pm \gamma \int_{\bbR^n} d^n x \, (u, x) |x|^{-3} |f(x)|^2, 
\quad f \in D^1(\bbR^n).      \lb{1.7} 
\end{split} 
\end{align}  
Here $\gamma_{c,n} > 0$ is optimal, that is, the largest possible such constant, and we recall that $D^1(\bbR^n)$ denotes the completion of $C_0^{\infty}(\bbR^n)$ with respect to the norm $\big( \int_{\bbR^n} d^n x \, |(\nabla g)(x)|^2\big)^{1/2}$, $g \in C_0^{\infty}(\bbR^n)$. 

The critical constant $\gamma_{c,n}$ can be characterized by the Rayleigh quotient
\begin{align}
\gamma_{c,n}^{-1} &= - \underset{f\in D^{1}(\bbR^{n})\setminus\{0\}}\sup\,
\left\{\f{\int_{\bbR^n} \,d^n x \, (u,x)|x|^{-3}|f(x)|^{2}}{\int_{\bbR^n} d^n x \, |\nabla f(x)|^{2}}\right\}    \lb{1.8} \\
&= \underset{\varphi \in H_{0}^{1}((0,\pi))\backslash\{0\}}{\sup} \Bigg\{\int_{0}^{\pi} d\theta_{n-1} \, 
[- \cos(\theta_{n-1})] |\varphi(\theta_{n-1})|^{2}     \no \\
& \quad \times \bigg[\int_{0}^{\pi} d\theta_{n-1} \, |\varphi'(\theta_{n-1})|^{2}+[(n-2)(n-4)/4][\sin(\theta_{n-1})]^{-2} 
|\varphi(\theta_{n-1})|^{2}\bigg]^{-1}\Bigg\},      \no\\
&\hspace*{9.6cm}  n\geq 3,        \lb{1.9}
\end{align} 
see \cite{FMT08}. To obtain \eqref{1.9} one introduces polar coordinates, 
\begin{equation}
x = r \omega, \quad  r=|x| \in (0,\infty), \quad \omega = \omega(\theta_1,\dots,\theta_{n-1}) = x/|x| \in \bbS^{n-1},     \lb{1.10} 
\end{equation} 
as in \eqref{A.1}--\eqref{A.4}. By \eqref{1.1}, clearly,
\begin{equation}
 (n-2)^2/4 \leq \gamma_{c,n}, \quad n \geq 3.     \lb{1.11}
\end{equation}
The existence of $\gamma_{c,n}$ as a finite positive number is shown for $n=3$ in \cite{HG80} and for $n \in \bbN$, 
$n \geq 3$, in \cite{FMT08}. 

Next, one decomposes 
\begin{align}
L_{\gamma} &=- \Delta+V_{\gamma}(x), \quad x \in \bbR^n\backslash\{0\},          \lb{1.12} \\
&= \bigg[- \f{d^{2}}{dr^{2}} - \f{n-1}{r}\f{d}{dr}\bigg]\otimes I_{L^{2}(\bbS^{n-1})} 
+ \f{1}{r^{2}}\otimes \Lambda_{\gamma,n}, \quad r\in (0,\infty), \; \gamma \geq 0, \; n \geq 3,  \no 
\end{align}
acting in $L^{2}((0,\infty);r^{n-1}dr)\otimes L^{2}(\bbS^{n-1})$, where
\begin{align}
\Lambda_{\gamma,n} = - \Delta_{\bbS^{n-1}} + \gamma\cos(\theta_{n-1}), \quad 
\dom(\Lambda_{\gamma,n}) = \dom(- \Delta_{\bbS^{n-1}}), \quad \gamma \geq 0, \; n \geq 3,     \lb{1.13}
\end{align}
with $- \Delta_{\bbS^{n-1}}$ the Laplace--Beltrami operator in $L^2(\bbS^{n-1})$ (see Appendix \ref{sA}). It is shown in 
\cite{FMT08} that $\gamma_{c,n}$ is also characterized by 
\begin{equation}
\lambda_{\gamma_{c,n}, n, 0} = - (n-2)^2/4, \quad n \geq 3,    \lb{1.14}
\end{equation}
where $\lambda_{\gamma, n, 0}$ denotes the lowest eigenvalue of $\Lambda_{\gamma,n}$. We rederive \eqref{1.14} via ODE methods and also prove in Theorem \ref{t3.1}, that 
\begin{equation}
\f{d \lambda_{\gamma, n, 0}}{d \gamma} \leq \f{\lambda_{\gamma, n, 0}}{\gamma} < 0, \quad \gamma > 0, \; n \geq 3     \lb{1.15}
\end{equation}
(this extends the $n=3$ result in \cite{HG80} to $n \in \bbN$, $n \geq 3$) and also provide the two-sided bounds 
\begin{equation}
- \f{\gamma^2}{(n-1)^2} \leq \lambda_{\gamma, n, 0} 
\leq - \f{\gamma}{2} \f{I_{n/2}(2\gamma/(n-1))}{I_{(n-2)/2}(2\gamma/(n-1))} < 0, \quad \gamma > 0, \; n \geq 3,   \lb{1.16} 
\end{equation}
with $I_{\nu}(\dott)$ the regular modified Bessel function of order $\nu \in \bbC$. (The additional lower bound 
$- \gamma \leq \lambda_{\gamma, n, 0}$ is of course evident since $-1 \leq \cos(\theta_{n-1})$.) Moreover, employing 
the fact that
\begin{equation}
- y'' (x) + \big\{\big[s^2 - (1/4)\big] \big/\sin^2(x)\big\} y(x) = z y(x), \quad z \in \bbC, \; s \in [0,\infty), \; x \in (0,\pi),     \lb{1.17} 
\end{equation}
is exactly solvable in terms of hypergeometric functions leads to inequality \eqref{3.41b}, and combining the latter with 
\eqref{1.9} enables us to prove the existence of $C_0 \in (0,\infty)$ such that 
\begin{equation}
\gamma_{c,n} \underset{n \to \infty}{=} C_0 (n-2)(n-4) [1 + \oh(1)],     \lb{1.18} 
\end{equation}
and the two-sided bounds 
\begin{equation}
15 \pi [(n-2)(n-4) + 4]/32 \geq \gamma_{c,n} \geq \begin{cases}
1/4, & n=3, \\
1, & n=4, \\
3^{3/2} [(n-2)(n-4) + 1]/8, & n \geq 5,
\end{cases}      \lb{1.19} 
\end{equation}
in Theorem \ref{t3.3}. 

Briefly turning to the content of each section, Section \ref{s2} offers a detailed treatment of the angular momentum decomposition of $H_{\gamma}$, the self-adjoint realization of $L_{\gamma} =- \Delta+V_{\gamma}(\dott)$ \big(i.e., the Friedrichs extension of $L_{\gamma}|_{C_0^{\infty}(\bbR^n\backslash\{0\})}$\big) in $L^2(\bbR^n)$, and hence together with Appendix \ref{sA} on spherical harmonics and the Laplace--Beltrami operator in $L^2(S^{n-1})$, $n \geq 2$, provides the background information for the bulk of this paper. Equations \eqref{1.15}, \eqref{1.16}, \eqref{1.18}, and \eqref{1.19} represent our principal new results in Section \ref{s3}. Section \ref{s4} develops a numerical approach to $\gamma_{c,n}$ that exhibits $- \gamma_{c,n}^{-1}$ as the smallest (negative) eigenvalue of a particular triangular operator $K(\gamma_{c,n})$ in $\ell^2(\bbN_0)$ with vanishing diagonal elements (cf.\ \eqref{4.19}, \eqref{4.21a}). In addition, we prove that finite truncations of $K(\gamma_{c,n})$ yield a convergent and efficient approximation scheme for $\gamma_{c,n}$. Finally, Section \ref{s5} considers the extension to multicenter dipole Hamiltonians of the form
\begin{equation}
L_{\{\gamma_j\}_{j \in J}} = - \Delta + \sum_{j \in J} \gamma_j \f{(u,(x-x_j))}{|x-x_j|^3} \chi_{B_n(x_j;\varepsilon/4)}(x) + W_0(x),
\end{equation}
where $J \subseteq \bbN$ is an index set, $\varepsilon > 0$, $\chi_{B_n(x_0;\eta)}$ denotes the characteristic function of the open ball $B_n(x_0;\eta) \subset \bbR^n$ with center $x_0 \in \bbR^n$ and radius $\eta > 0$, $\{x_j\}_{j \in J} \subset \bbR^n$, $n \in \bbN$, $n \geq 3$, and 
\begin{align}
\begin{split} 
& \inf_{j, j' \in J} |x_j - x_{j'}| \geq \varepsilon, \quad 0 \leq \gamma_j \leq \gamma_0 < \gamma_{c,n}, \; j \in J,   \\
&\, W_0 \in L^{\infty}(\bbR^n), \, \text{ $W_0$ real-valued a.e.~on $\bbR^n$.}
\end{split}
\end{align}
In particular, $\{x_j\}_{j \in J}$ is permitted to be an infinite, discrete set, for instance, a lattice. Relying on results proven in \cite{GMNT16}, we derive the optimal result that 
$L_{\{\gamma_j\}_{j \in J}}|_{C_0^{\infty}(\bbR^n\backslash \{x_j\}_{j \in J})}$ is bounded from below (resp., essentially self-adjoint) if each individual $L_{\gamma_j}|_{C_0^{\infty}(\bbR^n\backslash \{0\})}$, $j \in J$, is bounded from below (resp., essentially self-adjoint). This extends results in \cite{FMT09}, where $J$ is assumed to be finite.

\section{The Dipole Hamiltonian} \lb{s2}

In this section we provide a discussion of the angular momentum decomposition of the $n$-dimensional Laplacian $- \Delta$, introduce the dipole Hamiltonian $H_{\gamma}$, the principal object of this paper, and discuss an analogous decomposition of the latter. 

In spherical coordinates \eqref{A.1}, the Laplace differential expression in $n$ dimensions takes the form
\begin{align}
- \Delta = - \f{\p^{2}}{\p r^{2}}-\f{n-1}{r}\f{\p}{\p r}-\f{1}{r^{2}} \Delta_{\bbS^{n-1}}
	\lb{2.1}
\end{align}
where $-  \Delta_{\bbS^{n-1}}$ denotes the Laplace--Beltrami operator\footnote{We will call $-  \Delta$ the Laplacian to guarantee nonnegativity of the underlying $L^2(\bbR^n)$-realization (and analogously for the $L^2(\bbS^{n-1})$-realization of the Laplace--Beltrami operator 
$-  \Delta_{\bbS^{n-1}}$).} associated with the $(n-1)$-dimensional unit sphere 
$\bbS^{n-1}$ in $\bbR^n$, see \eqref{A.16}. When acting in $L^{2}(\bbR^n)$, which in spherical coordinates can be written as 
$L^{2}(\bbR^n) \simeq L^{2}((0,\infty);r^{n-1}dr)\otimes L^{2}(\bbS^{n-1})$, \eqref{2.1}  becomes 
\begin{align}
- \Delta = \bigg[- \f{d^{2}}{dr^{2}} - \f{n-1}{r}\f{d}{dr}\bigg]\otimes I_{L^{2}(\bbS^{n-1})} - \f{1}{r^{2}}\otimes  \Delta_{\bbS^{n-1}}      \lb{2.2}
\end{align}
(with $I_{\cX}$ denoting the identity operator in $\cX$).  
The Laplace--Beltrami operator $-  \Delta_{\bbS^{n-1}}$ in $L^{2}(\bbS^{n-1})$, with domain $\dom(- \Delta_{\bbS^{n-1}}) = H^2\big(\bbS^{n-1}\big)$ (cf., e.g., \cite{BLP19}),  is known to be essentially self-adjoint and nonnegative on $C_0^{\infty}(\bbS^{n-1})$ (cf.\ \cite[Theorem~5.2.3]{Da89}). Recalling the  treatment in \cite[p.~160--161]{RS75}, one decomposes the space $L^{2}(\bbS^{n-1})$ into an infinite orthogonal sum, yielding
\begin{align}
\begin{split} 
L^{2}(\bbR^n) &\simeq L^{2}((0,\infty);r^{n-1}dr)\otimes L^{2}(\bbS^{n-1})    \\
& = \bigoplus\limits_{\ell=0}^{\infty}L^{2}((0,\infty);r^{n-1}dr)\otimes\cY_{\ell}^{n},   \lb{2.3}
\end{split} 
\end{align}
where $\cY_{\ell}^{n}$ is the eigenspace of $- \Delta_{\bbS^{n-1}}$ corresponding to the eigenvalue $\ell(\ell+n-2)$, $\ell \in \bbN_0$, as 
\begin{equation}
\sigma (- \Delta_{\bbS^{n-1}}) = \{\ell(\ell+n-2)\}_{\ell \in \bbN_0}.     \lb{2.4} 
\end{equation}
In particular, this results in 
\begin{align}
- \Delta = \bigoplus\limits_{\ell=0}^{\infty} \left[-\f{d^{2}}{dr^{2}}-\f{n-1}{r}\f{d}{dr}
+ \f{\ell(\ell+n-2)}{r^{2}}\right] \otimes I_{\cY_{\ell}^{n}},       \lb{2.5}
\end{align}
in the space \eqref{2.3}.

To simplify matters, replacing the measure $r^{n-1}dr$ by $dr$ and simultaneously removing the term 
$(n-1)r^{-1} (d/dr)$, one introduces the unitary operator
\begin{align}
U_{n}=
\begin{cases}
L^{2}((0,\infty);r^{n-1}dr)\rightarrow L^{2}((0,\infty);dr),  \\[1mm] 
f(r)\mapsto r^{(n-1)/2}f(r),
\end{cases}
	\lb{2.6}
\end{align}
under which \eqref{2.5} becomes
\begin{align}
- \Delta = \bigoplus\limits_{\ell=0}^{\infty} U_{n}^{-1}\left[-\f{d^{2}}{dr^{2}} + 
\f{[(n-1)(n-3)/4] + \ell(\ell+n-2)}{r^{2}}\right] U_{n} \otimes I_{\cY_{\ell}^{n}}
	\lb{2.7}
\end{align}
acting in the space \eqref{2.3}. The precise self-adjoint $L^2$-realization of $- \Delta$ in the space 
\eqref{2.3} then is of the form
\begin{equation}
H_0 = \bigoplus\limits_{\ell=0}^{\infty} U_{n}^{-1} h_{n,\ell} \,U_{n} \otimes I_{\cY_{\ell}^{n}},
	\lb{2.8}
\end{equation}
where $h_{n,\ell}$, $\ell \in \bbN_0$, represents the Friedrichs extension of 
\begin{align}
\left[-\f{d^{2}}{dr^{2}} + 
\f{[(n-1)(n-3)/4] + \ell(\ell+n-2)}{r^{2}}\right]\bigg|_{C_0^{\infty}((0,\infty))}, 
\quad \ell \in \bbN_0,        \lb{2.9}
\end{align}
in $L^2((0,\infty); dr)$. For explicit operator domains and boundary conditions (the latter for $n=2,3$ only) we refer to 
\eqref{2.37}--\eqref{2.40}. It is well-known (cf.\ \cite[Sect.~IX.7, Appendix to X.1]{RS75}) that 
\begin{align}
& H_0 = - \Delta, \quad \dom(H_0) = H^2(\bbR^n),  \lb{2.13} \\
& H_0|_{C_0^{\infty}(\bbR^n)} \, \text{ is essentially self-adjoint,}     \lb{2.14}\\
& H_0|_{C_0^{\infty}(\bbR^n\backslash \{0\})} \, \text{ is essentially self-adjoint if and only if 
$n \geq 4$.}\lb{2.15}
\end{align}

Next, we turn to the dipole potential
\begin{align}
V_{\gamma}(x)=\gamma\f{(u,x)}{|x|^{3}}, \quad x\in\bbR^n \backslash \{0\}, \; \gamma \geq 0, \; n \geq 2, 
	\lb{2.16}
\end{align}
where $u \in \bbR^n$ is a unit vector in the direction of the dipole, the strength of the dipole equals 
$\gamma \geq 0$, and $(\, \cdot \,, \, \cdot \,)$ represents the Euclidean scalar product in $\bbR^n$. Upon an appropriate rotation, one can always choose the coordinate system in such a manner that 
$(u,x) = |x| \cos(\theta_{n-1})$, implying 
\begin{align}
V_{\gamma}(x)=\gamma\f{\cos(\theta_{n-1})}{|x|^{2}}, \quad x\in\bbR^n \backslash \{0\}, \; 
\gamma \geq 0, \; n \geq 2.
	\lb{2.17}
\end{align}
In the following we primarily restrict ourselves to the case $n \geq 3$ and comment on the exceptional case $n=2$ at the end of Section \ref{s3}. 
The differential expression associated with Hamiltonian for this system then becomes
\begin{align}
L_{\gamma}=- \Delta+V_{\gamma}(x), \quad x \in \bbR^n\backslash\{0\}, \; \gamma \geq 0, \; 
n \geq 3,   \lb{2.18}
\end{align}
acting in $L^{2}(\bbR^n)$. In analogy to \eqref{2.2}, \eqref{2.18} can be represented as
\begin{align}
L_{\gamma}= \bigg[- \f{d^{2}}{dr^{2}} - \f{n-1}{r}\f{d}{dr}\bigg]\otimes I_{L^{2}(\bbS^{n-1})} 
+ \f{1}{r^{2}}\otimes \Lambda_{\gamma,n}, \quad \gamma \geq 0, \; n \geq 3, 
	\lb{2.19}
\end{align}
acting in $L^{2}((0,\infty);r^{n-1}dr)\otimes L^{2}(\bbS^{n-1})$, where
\begin{align}
\Lambda_{\gamma,n} = - \Delta_{\bbS^{n-1}} + \gamma\cos(\theta_{n-1}), \quad 
\dom(\Lambda_{\gamma,n}) = \dom(- \Delta_{\bbS^{n-1}}), \quad \gamma \geq 0, \; n \geq 3, 
	\lb{2.20}
\end{align}
is self-adjoint in $L^{2}(\bbS^{n-1})$ (since $\gamma\cos(\theta_{n-1})$ is a bounded self-adjoint operator in $L^{2}(\bbS^{n-1})$). Applying the angular momentum decomposition to 
$L_{\gamma}$, but this time with respect to the eigenspaces of $\Lambda_{\gamma,n}$, then results in 
\begin{align}\begin{split} 
L^2(\bbR^n) &= L^2((0,\infty);r^{n-1}\,dr) \otimes L^2(\bbS^{n-1})    \\
& = \bigoplus\limits_{\ell=0}^{\infty}L^{2}((0,\infty);r^{n-1}\,dr)\otimes \cY_{\gamma,\ell}^{n}, 
\quad n \geq 3,  \lb{2.21}
\end{split} 
\end{align}
where $ \cY_{\gamma,\ell}^{n}$ represents the eigenspace of $\Lambda_{\gamma,n}$ corresponding to the eigenvalue $\lambda_{\gamma,n,\ell}$, as 
\begin{equation}
\sigma (\Lambda_{\gamma,n}) = \{\lambda_{\gamma,n,\ell}\}_{\ell \in \bbN_0}. 
	\lb{2.22}
\end{equation}
We will order the eigenvalues of $\Lambda_{\gamma,n}$ according to magnitude, that is,
\begin{equation}
\lambda_{\gamma,n,\ell} \leq \lambda_{\gamma,n,\ell + 1}, \quad \gamma \geq 0, \; 
\ell \in \bbN_0, \; n \geq 3,      \lb{2.23}
\end{equation}
repeating them according to their multiplicity. 
The analog of \eqref{2.7} in the space \eqref{2.21} then becomes 
\begin{align}
L_{\gamma}=\bigoplus\limits_{\ell=0}^{\infty}U_{n}^{-1} \left[-\f{d^{2}}{dr^{2}} 
+ \f{[(n-1)(n-3)/4] + \lambda_{\gamma,n,\ell}}{r^{2}} \right] U_{n} 
\otimes I_{ \cY_{\gamma,\ell}^{n}}, \quad n \geq 3.       \lb{2.24}
\end{align}

\begin{remark} \lb{r2.1} 
Since $e^{- t (-\Delta_{\bbS^{n-1}})}$, $t \geq 0$, has a continuous and nonnegative integral kernel (see, e,g., \cite[Theorem~5.2.1]{Da89}), it is positivity improving in $L^2(\bbS^{n-1})$. Hence, so is $e^{- t \Lambda_{\gamma,n}}$, $t \geq 0$, by (a special case of) \cite[Theorem~XIII.45]{RS78}. Thus, by \cite[Theorem~XIII.44]{RS78} one concludes that 
\begin{equation} 
\text{the lowest eigenvalue $\lambda_{\gamma,n,0}$ of $\Lambda_{\gamma,n}$ is simple for all $\gamma \geq 0$.}
	\lb{2.25}
\end{equation} 
${}$ \hfill $\diamond$  
\end{remark}

In order to deal exclusively with operators which are bounded from below we now make the the following assumption.

\begin{hypothesis} \lb{h2.2}
Suppose that $n \in \bbN$, $n \geq 3$, and $\gamma \geq 0$ are such that 
\begin{equation}
\lambda_{\gamma,n,0} \geq - (n-2)^2/4.    \lb{2.26}
\end{equation}
\end{hypothesis} 

Inequality \eqref{2.26} is inspired by Hardy's inequality \eqref{1.1} (cf.\ \cite[Sect.~1.2]{BEL15}, \cite[p.~345]{Ka95}, \cite[Ch.~3]{KMP07}, \cite[Ch.~1]{KPS17}, \cite[Ch.~1]{OK90}), which in turn implies
\begin{equation}
\bigg[- \f{d^2}{dr^2} + \f{c}{r^2}\bigg]\bigg|_{C_0^{\infty}((0,\infty))} \geq 0 \, \text{ if and only if 
$c \geq - 1/4$.}    \lb{2.27}
\end{equation}
In fact, ``$ \geq 0$'' in \eqref{2.27} can be replaced by ``bounded from below''. Assumption \eqref{2.26} is equivalent to 
\begin{equation}
[(n-1)(n-3)/4] + \lambda_{\gamma,n,0} \geq - 1/4.
	\lb{2.28}
\end{equation}

\begin{remark} \lb{r2.3} 
Since the perturbation $\gamma \cos(\theta_{n-1})$, $\gamma \in [0,\infty)$, of $-  \Delta_{\bbS^{n-1}}$ in \eqref{2.20} is bounded from below and from above,
\begin{equation}
- \gamma I_{L^2(\bbS^{n-1})} \leq \gamma \cos(\theta_{n-1}) \leq \gamma I_{L^2(\bbS^{n-1})}, 
	\lb{2.29}
\end{equation} 
and $-  \Delta_{\bbS^{n-1}} \geq 0$, it is clear that 
\begin{equation}
\lambda_{\gamma,n,0} \geq - \gamma, \, \text{ that is, } \, \Lambda_{\gamma,n} \geq - \gamma 
I_{L^2(\bbS^{n-1})},     \lb{2.30} 
\end{equation}
and $\lambda_{0,n,0} = 0$. 
In particular, for $n \geq 3$ and $0 \leq \gamma$ sufficiently small, Hypothesis \ref{h2.2} will be satisfied. We are particularly interested in the existence of a critical $\gamma_{c,n} > 0$ such that 
\begin{equation}
\lambda_{\gamma_{c,n},n,0} = - (n-2)^2/4,     \lb{2.31} 
\end{equation}
and whether or not 
\begin{equation}
\lambda_{\gamma,n,0} < - (n-2)^2/4, \quad \gamma \in (\gamma_{c,n}, \gamma_2), 
	\lb{2.32}
\end{equation}
for a $\gamma_2 \in (\gamma_{c,n}, \infty)$, with 
\begin{equation}
\lambda_{\gamma,n,0} \geq - (n-2)^2/4, \quad \gamma \in (\gamma_2, \gamma_3), 
	\lb{2.33}
\end{equation} 
for a $\gamma_3 \in (\gamma_2, \infty)$, etc. This will be clarified in the next section (demonstrating that 
$\gamma_2 = \infty$). \hfill $\diamond$
\end{remark}

Given Hypothesis \ref{h2.2}, the precise self-adjoint $L^2(\bbR^n)$-realization of $L_{\gamma}$ in the space 
\eqref{2.21} is then of the form 
\begin{equation}
H_{\gamma} = \bigoplus\limits_{\ell=0}^{\infty} U_{n}^{-1} h_{\gamma,n,\ell} \, 
U_{n} \otimes I_{ \cY_{\gamma,\ell}^{n}}, \quad \gamma \geq 0, \; n \geq 3, 
	\lb{2.34}
\end{equation}
where $h_{\gamma,n,\ell}$, $\ell \in \bbN_0$, represents the Friedrichs extension of 
\begin{align}
\left[-\f{d^{2}}{dr^{2}} + 
\f{[(n-1)(n-3)/4] + \lambda_{\gamma,n,\ell}}{r^{2}}\right]\bigg|_{C_0^{\infty}((0,\infty))}, 
\quad r > 0, \; \gamma \geq 0, \; n \geq 3, \; \ell \in \bbN_0,      \lb{2.35}
\end{align}
in $L^2((0,\infty); dr)$. Explicitly, as discussed, for instance, in \cite{GLN20}, \cite{GP79}, the Friedrichs extension of 
$h_{\gamma,n,\ell}$, $\ell \in \bbN_0$, can be determined from the fact that  the Friedrichs extension 
$h_{\alpha,F}$ in $L^2((0,\infty); dr)$ of 
\begin{equation}
h_{\alpha} = \bigg[- \f{d^2}{dr^2} + \f{\alpha^2 -(1/4)}{r^2}\bigg]\bigg|_{C_0^{\infty}((0,\infty))}, 
\quad r> 0, \; \alpha \in [0,\infty), 
	\lb{2.36}
\end{equation}
is given by 
\begin{align}
&h_{\alpha,F} = - \f{d^2}{dr^2} + \f{\alpha^2 -(1/4)}{r^2}, \quad r> 0, \; \alpha \in [0,\infty),   \lb{2.37} \\
& \dom\big(h_{\alpha,F}\big) = \big\{f \in L^2((0,\infty); dr) \,\big|\, f, f' \in AC_{loc}((0,\infty)); \, 
\wti f_{\alpha}(0) = 0;    \lb{2.38} \\
& \hspace*{2.5cm}  (-f'' + \big[\alpha^2 - (1/4)\big]r^{-2}f) \in L^2((0,\infty); dr)\big\}, 
\quad \alpha \in [0,1),     \no \\
& \dom\big(h_{\alpha,F}\big) = \big\{f \in L^2((0,\infty); dr) \,\big|\, f, f' \in 
AC_{loc}((0,\infty));       \lb{2.39} \\ 
& \hspace*{2.5cm} (-f'' +\big[\alpha^2 - (1/4)\big]r^{-2}f) \in L^2((0,\infty); dr)\big\}, \quad \alpha \in [1,\infty),     \no 
\end{align}
where
\begin{equation}
\wti f_{\alpha}(0) = \begin{cases} \lim_{r \downarrow 0} 
\big[r^{1/2} \ln(1/r)\big]^{-1} f(r), & \alpha =0, \\[1mm] 
\lim_{r \downarrow 0} 2 \alpha r^{\alpha - (1/2)}f(r), & \alpha \in (0,1).  
\end{cases}      \lb{2.40}
\end{equation}

Next we note the following fact.

\begin{lemma} \lb{l2.4}
Given the operator $\Lambda_{\gamma,n}$, $\gamma \geq 0$, in $L^{2}(\bbS^{n-1})$ as introduced in 
\eqref{2.20}, one infers that 
\begin{equation} 
\lim_{\gamma\downarrow 0} \lambda_{\gamma,n,\ell} = \ell(\ell+n-2), \quad \ell \in \bbN_0,
	\lb{2.41}
\end{equation} 
recalling that $\{\ell(\ell+n-2)\}_{\ell \in \bbN_0}$ are the corresponding eigenvalues of the unperturbed operator, $\Lambda_{0,n} = - \Delta_{\bbS^{n-1}}$, the Laplace--Beltrami operator $($cf.\ \eqref{2.4}$)$.  
\end{lemma}
\begin{proof}
This is a special case of Rellich's theorem in the form recorded, for instance, in 
\cite[Theorems~XII.3 and XII.13]{RS78}.
\end{proof}

\begin{lemma} \lb{l2.5}
Assume Hypothesis \ref{h2.2}, that is, suppose that 
\begin{equation}
\lambda_{\gamma,n,0} \geq - (n-2)^2/4, \quad n \geq 3. 
\end{equation}
Then $H_{\gamma}$ has purely absolutely continuous spectrum, 
\begin{equation}
\sigma(H_{\gamma}) = \sigma_{ac}(H_{\gamma}) = [0,\infty). 
	\lb{2.42}
\end{equation}
\end{lemma}
\begin{proof}
First, one notes that $H_{\gamma}$ is bounded from below if and only if each 
$h_{\gamma,n,\ell}$, $\ell \in \bbN_0$, is bounded from below. The ordinary differential operators 
$h_{\gamma,n,\ell}$, $\ell \in \bbN_0$, are well-known to have purely absolutely continuous spectrum 
equal to $[0,\infty)$, as proven, for instance in \cite{EK07} and \cite{GZ06}. Thus the result follows from the special case of direct sums (instead of direct integrals) in \cite[Theorem~XIII.85\,(f)]{RS78}. 
\end{proof}

\section{Criticality} \lb{s3}

We now turn to one of the principal questions -- a discussion of which $\gamma \geq 0$ cause 
$H_{\gamma}$ to be bounded from below. 

The natural space to which Hardy's inequality and its analog in connection with a dipole potential extends is the space $D^1(\bbR^n)$ (sometimes also denoted $D_0^1(\bbR^n)$, or $D^{1,2}(\bbR^n)$) obtained as the closure of $C_0^{\infty}(\bbR^n)$ with respect to the gradient norm, 
\begin{align}
& D^1(\bbR^n) = \ol{C_0^{\infty}(\bbR^n)}^{\|\, \cdot\,\|_{\nabla}},  \quad \|f\|_{\nabla} = \bigg(\int_{\bbR^n} d^n x \, |(\nabla f)(x)|^2\bigg)^{1/2}, \quad f \in C_0^{\infty}(\bbR^n),
\end{align}
see also \cite[pp.~201--204]{LL01}.

\begin{theorem} \lb{t3.1} 
Assume Hypothesis \ref{h2.2}. Then for all $n\geq 3$, there exists a unique critical dipole moment 
$\gamma_{c,n}>0$ characterized by 
\begin{equation}
\lambda_{\gamma_{c,n},n,0} = - (n-2)^2/4
\end{equation}
$($cf.\ \eqref{2.31} in Remark \ref{r2.3}$)$. Moreover, 
$\lambda_{\gamma,n,0}$ is strictly monotonically decreasing with respect to $\gamma \in (0,\infty)$, 
$\lambda_{0,n,0} = 0$, and 
\begin{equation}
\f{d\lambda_{\gamma,n,0}}{d\gamma}\leq\f{\lambda_{\gamma,n,0}}{\gamma} < 0 \, 
\text{ as well as } \, \lambda_{\gamma,n,0} \geq - \gamma, \quad \gamma \in (0,\infty).      \lb{3.1}
\end{equation}
Moreover, 
\begin{equation}
- \f{\gamma^2}{(n-1)^2} \leq \lambda_{\gamma,n,0} \leq 
- \f{\gamma}{2} \f{I_{n/2}(2 \gamma/(n-1))}{I_{(n-2)/2}(2 \gamma/(n-1))} < 0, \quad \gamma \in (0,\infty), 
\end{equation}
hold. In particular, $H_{\gamma}$ is bounded from below, and then $H_{\gamma} \geq 0$, if and only if 
$\gamma \in [0, \gamma_{c,n}]$. Consequently,  
\begin{align} 
\begin{split} 
&\text{for all $\gamma \in [0,\gamma_{c,n}]$, and all $u \in \bbR^n$, $|u|=1$,} \\
&\quad \int_{\bbR^n} d^n x \, |(\nabla f)(x)|^2 \geq \pm \gamma \int_{\bbR^n} d^n x \, (u, x) |x|^{-3} |f(x)|^2, 
\quad f \in D^1(\bbR^n).      \lb{3.2}
\end{split} 
\end{align}  
The constant $\gamma_{c,n} > 0$ in \eqref{3.2} is optimal $($i.e., the largest possible\,$)$, in addition, 
\begin{equation}
\gamma_{c,n} \geq (n-2)^2/4.    \lb{3.2a} 
\end{equation}
Finally, 
\begin{equation}
\sigma(H_{\gamma}) = \sigma_{ac}(H_{\gamma}) = [0,\infty), \quad \gamma \in [0, \gamma_{c,n}].     \lb{3.2b}
\end{equation}
\end{theorem}
\begin{proof} 
Existence of some critical dipole moment $\gamma_{c,n} > 0$ is clear from the discussion in 
Remark \ref{r2.3}. To prove the remaining claims regarding $\lambda_{\gamma,n,0}$ in Theorem \ref{t3.1}, we seek spherical harmonics dependent only on the final angle $\theta_{n-1}$, as this is the only angular variable dependence of $V_{\gamma}(\dott)$. From \eqref{A.10}--\eqref{A.13}, one infers these are precisely the ones indexed by the particular multi-indices $(\ell,0,\ldots,0)\in\bbN_{0}^{n}$, that is (cf.\ \eqref{A.10}--\eqref{A.14}),
\begin{align}
\begin{split} 
Y_{(\ell,0,\ldots,0)}(\theta_{n-1})=\bigg[\f{[(n-2)/2](n-2)_{\ell}}{\ell!(\ell+[(n-2)/2])}\bigg]^{1/2} 
C_{\ell}^{(n-2)/2}(\cos(\theta_{n-1})),&      \lb{3.3} \\[1mm] 
\ell \in \bbN_0, \; \theta_{n-1}\in[0,\pi).&
\end{split} 
\end{align}
Introducing the subspace
\begin{align}
\cL^{n} = {\rm lin.span} \{Y_{(\ell,0,\ldots,0)}\}_{\ell\in\bbN_{0}}.
	\lb{3.4}
\end{align} 
and restricting the Laplace--Beltrami differential expression \eqref{A.16} to $\cL^{n}$, one finds for \eqref{2.20}, 
\begin{align}
\Lambda_{\gamma,\cL^{n}} &= -\f{d^{2}}{d\theta_{n-1}^{2}} - (n-2)\cot(\theta_{n-1})\f{d}{d\theta_{n-1}} + \gamma\cos(\theta_{n-1}), \quad \theta_{n-1} \in (0,\pi),   \lb{3.5} 
\end{align}
acting on functions in $L^2\big((0,\pi); [\sin(\theta_{n-1})]^{n-2} d \theta_{n-1}\big)$. Reverting from the weighted measure $[\sin(\theta_{n-1})]^{n-2} d \theta_{n-1}$ to Lebesgue measure $d \theta_{n-1}$ on $(0,\pi)$ in a unitary fashion then yields the differential expression $\wti \Lambda_{\gamma,\cL^{n}}$ given by 
\begin{equation}
\wti \Lambda_{\gamma,\cL^{n}} = - \f{d^2}{d \theta_{n-1}^2} + \f{(n-2)(n-4)}{4 \sin^2(\theta_{n-1})} - \f{(n-2)^2}{4} 
+ \gamma \cos(\theta_{n-1}), \quad  \theta_{n-1} \in (0,\pi),    \lb{3.7} 
\end{equation}
now acting on functions in $L^{2}((0,\pi);d\theta_{n-1})$.

Next, introducing the change of variable $\xi=\cos(\theta_{n-1}) \in (-1,1)$, $\Lambda_{\gamma,\cL^{n}}$ in \eqref{3.5} turns into 
\begin{align}
\begin{split} 
{\ul \Lambda}_{\gamma,\cL^{n}} 
&= \big(1-\xi^2\big)^{- (n-3)/2} \bigg[- \f{d}{d \xi}  \big(1-\xi^2\big)^{(n-1)/2} \f{d}{d \xi} 
+ \gamma \big(1-\xi^2\big)^{(n-3)/2} \xi\bigg],        \lb{3.9} \\
& \hspace*{8.4cm} \xi \in (-1,1),     
\end{split} 
\end{align}
acting on functions in $L^{2}\Big((-1,1); \big(1-\xi^2 \big)^{(n-3)/2} d\xi \Big)$. We also note that reverting from the weighted measure $\big(1-\xi^2 \big)^{(n-3)/2} d\xi$ to Lebesgue measure $d \xi$ on $(-1,1)$ in a unitary fashion then finally yields the differential expression $\wti {\ul \Lambda}_{\gamma,\cL^{n}}$ given by 
\begin{align}
\begin{split} 
\wti {\ul \Lambda}_{\gamma,\cL^{n}} 
&= - \f{d}{d \xi}  \big(1-\xi^2\big)^{(n-1)/2} \f{d}{d \xi} 
+ \f{(n-3)^2}{4 \big(1-\xi^2\big)} - \f{(n-1)(n-3)}{4} + \gamma \, \xi,    \lb{3.12} \\
& \hspace*{7.4cm} \quad \xi \in (-1,1),      
\end{split} 
\end{align} 
acting on functions in $L^2((-1,1); d\xi)$. One observes that the first two terms on the right-hand side of \eqref{3.12} represent the Legendre operator $L_{\mu}$ in $L^2((-1,1); d\xi)$ associated with the differential expression
\begin{equation}
L_{\mu} = - \f{d}{d \xi}  \big(1-\xi^2\big)^{(n-1)/2} \f{d}{d \xi} 
+ \f{\mu^2}{\big(1-\xi^2\big)}, \quad \mu \in [0,\infty), \; \xi \in (-1,1),      \lb{3.12a} 
\end{equation}
which is in the limit circle case at $\pm 1$ if $\mu \in [0,1)$ and in the limit point case at $\pm 1$ if $\mu \in [1,\infty)$, as discussed in detail in \cite{EL05}. In particular, applying this fact to $\Lambda_{\gamma,\cL^{n}}$, 
$\wti \Lambda_{\gamma,\cL^{n}}$, ${\ul \Lambda}_{\gamma,\cL^{n}}$, and $\wti {\ul \Lambda}_{\gamma,\cL^{n}}$ yields the necessity of the Friedrichs boundary condition for $n=3, 4$, whereas for $n \in \bbN$, $n \geq 5$,  
$\Lambda_{\gamma,\cL^{n}}$ and $\wti \Lambda_{\gamma,\cL^{n}}$ (resp., ${\ul \Lambda}_{\gamma,\cL^{n}}$ and 
$\wti {\ul \Lambda}_{\gamma,\cL^{n}}$) are essentially self-adjoint on $C_0^{\infty}((0,\pi))$ (resp., 
$C_0^{\infty}((-1,1))$) and hence the associated maximally defined operators are self-adjoint. For the explicit form of the Friedrichs boundary condition corresponding to \eqref{3.12a} and hence \eqref{3.12} we also refer to \cite{EL05}.  
Due to the $\theta_{n-1}^{-2}$ (resp., $(\pi - \theta_{n-1})^{-2}$) singularity at $\theta_{n-1} = 0$ (resp., 
$\theta_{n-1} = \pi$), the Friedrichs extension corresponding to $\wti \Lambda_{\gamma,\cL^{n}}$ in \eqref{3.7} is clear from \eqref{2.36}--\eqref{2.40}. 

Following \cite{HG80} in the special case $n=3$, choosing $\psi \in  \dom \big( {\ul \Lambda}_{\gamma,\cL^{n}}\big)$ normalized, 
\begin{equation}
\|\psi\|_{L^{2}((-1,1); (1-\xi^2)^{(n-3)/2} d\xi)} = 1, 
\end{equation}
an appropriate integration by parts yields    
\begin{align}
& (\psi, {\ul \Lambda}_{\gamma,\cL^{n}} \psi)_{L^{2}((-1,1); (1-\xi^2)^{(n-3)/2} d\xi)}    \no \\
& \quad = \int_{-1}^{1} d\xi \Big[\big(1 - \xi^2\big)^{(n-1)/2} |\psi'(\xi)|^2  
+ \gamma \big(1 - \xi^2\big)^{(n-3)/2} \xi  |\psi(\xi)|^2\Big]    \lb{3.12c} \\
& \quad = \int_{-1}^{1} d\xi \big(1 - \xi^2\big)^{(n-1)/2} \Big[\big|\psi'(\xi) + \gamma (n-1)^{-1} \psi(\xi)\big|^2  
- \gamma^2 (n-1)^{-2} |\psi(\xi)|^2\Big]      \lb{3.12d} \\
& \quad \geq - \f{\gamma^2}{(n-1)^2} \int_{-1}^{1} d\xi \big(1 - \xi^2\big)^{(n-1)/2} |\psi(\xi)|^2    \no \\
& \quad \geq - \f{\gamma^2}{(n-1)^2} \int_{-1}^{1} d\xi \big(1 - \xi^2\big)^{(n-3)/2} |\psi(\xi)|^2   
= - \f{\gamma^2}{(n-1)^2}.
\end{align}
In particular, choosing for $\psi$ a normalized eigenfunction of ${\ul \Lambda}_{\gamma,\cL^{n}}$ corresponding to 
the eigenvalue $\lambda_{\gamma,n,0}$ in \eqref{3.12c} implies the lower bound
\begin{equation}
\lambda_{\gamma,n,0} \geq - \gamma^2\big/(n-1)^2.    \lb{3.12e}
\end{equation}
On the other hand (following once more  \cite{HG80} in the special case $n=3$), employing the normalized trial function (cf.\ \cite[no.~3.387]{GR80})
\begin{align}
& \phi_{\gamma} (\xi) = C_{\gamma} \, e^{- \gamma (n-1)^{-1} \xi}, \quad \xi \in (-1,1),     \no \\
& C_{\gamma} = \pi^{-1/4} [\gamma/(n-1)]^{(n-2)/4} [\Gamma((n-1)/2)]^{-1/2} [I_{(n-2)/2}(2\gamma/(n-1))]^{-1/2},    \\
& \|\phi_{\gamma}\|_{L^{2}((-1,1); (1-\xi^2)^{(n-3)/2} d\xi)} = 1,    \no      
\end{align}
with $I_{\nu}(\dott)$ the regular modified Bessel function of order $\nu \in \bbC$ (cf.\ \cite[Sect.~9.6]{AS72}), 
an application of the min/max principle and \eqref{3.12d} yield the upper bound 
\begin{align}
\lambda_{\gamma,n,0} & \leq 
(\phi_{\gamma}, {\ul \Lambda}_{\gamma,\cL^{n}}\phi_{\gamma})_{L^{2}((-1,1); (1-\xi^2)^{(n-3)/2} d\xi)}    \no\\
&= - \f{\gamma^2}{(n-1)^2} \int_{-1}^1 d\xi \, \big(1 - \xi^2\big)^{(n-1)/2} \phi_{\gamma}(\xi)^2     \no \\ 
& = - \f{\gamma}{2} \f{I_{n/2}(2 \gamma/(n-1))}{I_{(n-2)/2}(2 \gamma/(n-1))} < 0, \quad \gamma \in (0,\infty),      \lb{3.13}
\end{align}
employing \cite[no.~3.387]{GR80} once again. Thus, \eqref{3.13} implies that 
\begin{equation} 
\lambda_{\gamma,n,0} < 0, \quad \gamma \in (0,\infty),     \lb{3.13a} 
\end{equation}
and one infers a quadratic upper bound as $\gamma \downarrow 0$
\begin{equation}
- \f{\gamma}{2} \f{I_{n/2}(2 \gamma/(n-1))}{I_{(n-2)/2}(2 \gamma/(n-1))} \underset{\gamma \downarrow 0}{=} 
- \f{\gamma^2}{n(n-1)} \big[1 + \Oh\big(\gamma^2\big)\big],   
\end{equation}
in addition to the quadratic lower bound in \eqref{3.12e}. 

Next, recalling that the lowest eigenvalue $\lambda_{\gamma,n,0}$ of $\Lambda_{\gamma,n}$ is simple for all $\gamma \geq 0$ (and is also the lowest eigenvalue of $\Lambda_{\gamma,\cL^{n}}$, 
$\wti \Lambda_{\gamma,\cL^{n}}$, and ${\ul \Lambda}_{\gamma,\cL^{n}}$), we denote by $\psi_{\gamma, 0} \in \dom(\Lambda_{\gamma,n})$ the corresponding normalized eigenfunction, that is,
\begin{equation}
\Lambda_{\gamma,n} \psi_{\gamma, 0} = \lambda_{\gamma,n,0} \psi_{\gamma, 0}, \quad 
\|\psi_{\gamma, 0}\|_{L^{2}(\bbS^{n-1})} =1, \quad \gamma \in [0,\infty). 
\lb{3.14} 
\end{equation} 
Thus, one gets
\begin{align}
\begin{split}
\lambda_{\gamma,n,0}&=(\psi_{\gamma, 0}, \Lambda_{\gamma,n} 
\psi_{\gamma, 0})_{L^{2}(\bbS^{n-1})}     \\
&=(\psi_{\gamma, 0},[- \Delta_{\bbS^{n-1}}+\gamma\cos(\theta_{n-1})] 
\psi_{\gamma, 0})_{L^{2}(\bbS^{n-1})}.
\end{split}
      \lb{3.15}
\end{align}

Moreover, one observes that $\{\Lambda_{\gamma,n}\}_{\gamma\in[0,\infty)}$ is a self-adjoint analytic (in fact, entire) family of type $(A)$ in the sense of Kato (cf.\ \cite[Sect.~VII.2, p.~375--379]{Ka95}, \cite[p.~16]{RS78}), implying analyticity of  $\lambda_{\gamma,n,0}$ and $ \psi_{\gamma, 0}$ with respect to $\gamma$ in a complex neighborhood of $[0,\infty)$. In particular, $\lambda_{\gamma,n,0}$ is differentiable with respect 
to $\gamma$, and the Feynman--Hellmann Theorem \cite[p.~151]{Th81} (see also \cite[Theorem~1.4.7]{Si15}) yields  that
\begin{align}
\f{d\lambda_{\gamma,n,0}}{d\gamma}=(\psi_{\gamma, 0},\cos(\theta_{n-1}) 
\psi_{\gamma, 0})_{L^{2}(\bbS^{n-1})}, \quad \gamma \in (0,\infty).
	\lb{3.16}
\end{align}
Returning to the discussion of \eqref{2.30} in Remark \ref{r2.3}, employing $- \Delta_{\bbS^{n-1}} \geq 0$, one obtains 
\begin{equation}
\lambda_{\gamma,n,0} = (\psi_{\gamma, 0}, \Lambda_{\gamma,n}\psi_{\gamma, 0})_{L^{2}(\bbS^{n-1})} 
 \geq (\psi_{\gamma, 0},\gamma\cos(\theta_{n-1})\psi_{\gamma, 0})_{L^{2}(\bbS^{n-1})} \geq - \gamma, 
        \lb{3.17}
\end{equation}
implying,  
\begin{align}
\f{d\lambda_{\gamma,n,0}}{d\gamma} 
= (\psi_{\gamma, 0},\cos(\theta_{n-1})\psi_{\gamma, 0})_{L^{2}(\bbS^{n-1})} 
\leq \f{\lambda_{\gamma,n,0}}{\gamma} < 0, \quad \gamma \in (0,\infty),
	\lb{3.18}
\end{align}
by the strict negativity of $\lambda_{\gamma,n,0}$ for $\gamma > 0$ derived in \eqref{3.13}.

Given the existence of a unique critical dipole moment $\gamma_{c,n} > 0$ one concludes from \eqref{2.27},  \eqref{2.34}, and \eqref{2.35} the following fact: 
\begin{align}
\begin{split} 
& H_{\gamma}|_{C_0^{\infty}(\bbR^n \backslash \{0\})} \, \text{ is bounded from below, in fact, nonnegative,} 
\\
& \quad \text{if and only if } \, \gamma \in [0, \gamma_{c,n}], 
\end{split} 
\end{align}
and an integration by parts thus yields 
\begin{equation}
\pm \gamma \int_{\bbR^n} d^nx \, (u,x) |x|^{-3} |g(x)|^2 \leq \int_{\bbR^n} d^n x \, |(\nabla g)(x)|^2, 
\quad g \in C_0^{\infty}(\bbR^n), \quad \gamma \in [0, \gamma_{c,n}].    \lb{3.24AA} 
\end{equation}
It remains to extend \eqref{3.24AA} to elements $f \in D^1(\bbR)$. As in the case of the Hardy inequality \eqref{1.1}, this follows from invoking a Fatou-type argument to be outlined next. 

Since $C_0^{\infty}(\bbR^n)$ is dense in $D^1(\bbR^n)$, given $f \in D^1(\bbR^n)$ we pick a sequence 
$\{f_j\}_{j \in \bbN} \subset C_0^{\infty}(\bbR^n)$ such that $\lim_{j \to\infty}\|f_j - f\|_{D^1(\bbR^n)} = 0$, 
and, by passing to a subsequence, we may assume without loss of generality (see \eqref{3.26a} below) that $f_j \underset{j \to \infty}{\longrightarrow} f$ a.e.~on $\bbR^n$. 
(For the remainder of this proof $f, f_j$, $j \in \bbN$, will always be assumed to have the properties just discussed.) 
Indeed, the Sobolev inequality (see, e.g., \cite[Theorem~8.3]{LL01}, \cite{Ta76}),
\begin{align}
\begin{split} 
&\|\nabla f\|_{L^2(\bbR^n)}^2 \geq S_n \|f\|_{L^{2^*}(\bbR^n)}^2, \quad f \in D^1(\bbR^n), 
\quad 2^* = 2n/(n-2),    \lb{3.26a} \\
& \, S_n = [n(n-2)/4] 2^{2/n} \pi^{(n+1)/n} \Gamma((n+1)/2)^{-2/n},  \quad n \geq 3, 
\end{split} 
\end{align}
($\Gamma(\, \cdot \,)$ the Gamma function, cf.\ \cite[Sect.~6.1]{AS72}), yields convergence of $f_j$ to $f$ in $L^{2^*}(\bbR^n)$ and hence permits the selection of a subsequence that converges pointwise a.e.~Thus, given Hardy's inequality for functions in $C_0^{\infty}(\bbR^n)$, a well-known fact (see, e.g., \cite[Corollary~1.2.6]{BEL15}), 
\begin{equation}
\big[(n-2)^2/4\big] \int_{\bbR^n} d^nx \, |x|^{-2} |g(x)|^2 \leq \int_{\bbR^n} d^n x \, |(\nabla g)(x)|^2, 
\quad g \in C_0^{\infty}(\bbR^n),     \lb{3.22} 
\end{equation}
one obtains, 
\begin{align}
& \big[(n-2)^2/4\big] \int_{\bbR^n} d^nx \, |x|^{-2} |f_j(x)|^2 
\leq \int_{\bbR^n} d^n x \, |[\nabla (f_j - f + f )](x)|^2   \no \\
& \quad \leq 2 \int_{\bbR^n} d^n x \, |[\nabla (f_j - f )](x)|^2 + 2 \int_{\bbR^n} d^n x \, |(\nabla f )(x)|^2 \leq C,   
\end{align}
for some $C \in (0,\infty)$ independent of $j \in \bbN$. Thus, 
\begin{equation} 
\big[(n-2)^2/4\big] \int_{\bbR^n} d^n x \, |x|^{-2} |f(x)|^2 \leq C, \quad f \in D^1(\bbR^n),    \lb{3.23} 
\end{equation} 
by a consequence of Fatou's Lemma (see, e.g., \cite[p.~21]{LL01}. Hence,  
\begin{align}
& \big[(n-2)^2/4\big] \int_{\bbR^n} d^nx \, |x|^{-2} |f(x)|^2 
= \big[(n-2)^2/4\big] \int_{\bbR^n} d^nx \, \lim_{j \to \infty} |x|^{-2} |f_j(x)|^2  \no \\
& \quad = \big[(n-2)^2/4\big] \int_{\bbR^n} d^nx \, \liminf_{j \to \infty}  |x|^{-2} |f_j(x)|^2  \no \\
& \quad \leq  \big[(n-2)^2/4\big]  
\liminf_{j \to \infty} \int_{\bbR^n} d^nx \, |x|^{-2} |f_j(x)|^2  \quad \text{(by Fatou's Lemma)} 
\no \\
&\quad \leq \liminf_{j \to \infty} \int_{\bbR^n} d^nx \, |(\nabla f_j)(x)|^2  \quad \text{(by \eqref{3.22})}   \no \\
& \quad = \lim_{j \to \infty} \int_{\bbR^n} d^nx \, |(\nabla f_j)(x)|^2 = \int_{\bbR^n} d^nx \, |(\nabla f)(x)|^2,    \lb{3.24}
\end{align} 
extends Hardy's inequality \eqref{3.22} from $C_0^{\infty} (\bbR^n)$ to $D^1(\bbR^n)$. Hardy's inequality on $D^1(\bbR^n)$ also implies that
\begin{equation}
\lim_{j \to \infty} \int_{\bbR^n} d^n x \, |x|^{-2} |f(x) - f_j(x)|^2 = 0,    \lb{3.24A} 
\end{equation}
in particular,
\begin{equation}
\lim_{j \to \infty} \int_{\bbR^n} d^n x \, |x|^{-2} |f_j(x)|^2 = \int_{\bbR^n} d^n x \, |x|^{-2} |f(x)|^2.  \lb{3.24a}
\end{equation}
Since 
\begin{equation}
|(u,x)| |x|^{-1} \leq 1, \quad x \in \bbR^n \backslash \{0\}, \quad u \in \bbR^n, \; |u| = 1,  \lb{3.24B} 
\end{equation}
\eqref{3.23} also implies 
\begin{equation} 
\big[(n-2)^2/4\big] \int_{\bbR^n} d^n x \, |(u,x)| |x|^{-3} |f(x)|^2 \leq C, \quad f \in D^1(\bbR^n),    \lb{3.24b} 
\end{equation} 
similarly, \eqref{3.24A}, \eqref{3.24a}, and H\"older's inequality imply
\begin{equation}
\lim_{j \to \infty} \int_{\bbR^n} d^n x \, (u,x) |x|^{-3} |f_j(x)|^2 
= \int_{\bbR^n} d^n x \, (u,x) |x|^{-3} |f(x)|^2.  \lb{3.24c}
\end{equation}
Thus, for $\gamma \in [0,\gamma_{c,n}]$,
\begin{align}
& \pm \gamma \int_{\bbR^n} d^nx \, (u,x) |x|^{-3} |f(x)|^2 = \pm \lim_{j \to \infty} \gamma 
\int_{\bbR^n} d^nx \, (u,x) |x|^{-3} |f_j(x)|^2  \quad \text{(by \eqref{3.24c})} 
\no \\
&\quad \leq \lim_{j \to \infty} \int_{\bbR^n} d^nx \, |(\nabla f_j)(x)|^2  \quad \text{(by \eqref{3.24AA})}    \no \\ 
& \quad = \int_{\bbR^n} d^nx \, |(\nabla f)(x)|^2,    \lb{3.24d}
\end{align} 
finally implying \eqref{3.2}. Moreover, \eqref{3.24B} also yields 
\begin{align}
\int_{\bbR^n} d^n x \, |(\nabla f)(x)|^2 &\geq [(n-2)/2]^2 \int_{\bbR^n} d^n x \, |x|^{-2} |f(x)|^2   \lb{3.37a} \\
& \geq [(n-2)/2]^2 \int_{\bbR^n} d^n x \, |(u,x)| |x|^{-3} |f(x)|^2, \quad  f \in D^1(\bbR^n),   \no 
\end{align}
and hence \eqref{3.2a}. 

Finally, \eqref{3.2b} is clear from Lemma \ref{l2.5} and the strict monotonicity of $\lambda_{\gamma,n,0}$ with respect to $\gamma \geq 0$. 
\end{proof}

\begin{remark} \lb{r3.2} 
$(i)$ Theorem \ref{t3.1} demonstrates that $\gamma_2 = \infty$ in Remark \ref{r2.3}. \\[1mm]
$(ii)$ Inequality \eqref{3.2a}, that is, $\gamma_{c,n} \geq (n-2)^2/4$, shows that $\gamma_{c,n}$ grows at least like $c n^2$ for appropriate $ c >0$ as $n \to \infty$. \hfill $\diamond$
\end{remark}

Next, we improve upon Remark \ref{r3.2}\,$(ii)$ for $n \geq 5$ as follows:

\begin{theorem} \lb{t3.3}
Assume Hypothesis \ref{h2.2}. Then there exists $C_0 \in (0,\infty)$ such that 
\begin{equation}
\gamma_{c,n} \underset{n \to \infty}{=} C_0 (n-2)(n-4) [1 + \oh(1)],     \lb{3.38} 
\end{equation}
in addition, 
\begin{equation}
15 \pi [(n-2)(n-4) + 4]/32 \geq \gamma_{c,n} \geq \begin{cases}
1/4, & n=3, \\
1, & n=4, \\
3^{3/2} [(n-2)(n-4) + 1]/8, & n \geq 5,
\end{cases}      \lb{3.41e} 
\end{equation}
\end{theorem} 
\begin{proof} 
Employing \cite[eq.~(1) and Remark~1]{FMT08}, one considers the Rayleigh quotient
\begin{align}
\begin{split} 
\Gamma_{n}\big(\gamma(u,x)|x|^{-1}\big)= -\gamma \underset{f\in D^{1}(\bbR^{n})\setminus\{0\}}\sup\,
\left\{\f{\int_{\bbR^n} \,d^n x \, (u,x)|x|^{-3}|f(x)|^{2}}{\int_{\bbR^n} d^n x \, |\nabla f(x)|^{2}}\right\},& \\ 
\gamma\in(0,\infty), \; n\geq 3,&     \lb{3.40a}
\end{split} 
\end{align}
and notes that $\Gamma_{n}\big(\gamma(u,x)|x|^{-1}\big) \uparrow 1$ as $\gamma \uparrow \gamma_{c,n}$, implying 
(cf.\ \eqref{3.7})  
\begin{align}
\gamma_{c,n}^{-1} &= \underset{\varphi \in H_{0}^{1}((0,\pi))\backslash\{0\}}{\sup} \Bigg\{\int_{0}^{\pi} d\theta_{n-1} \, 
[- \cos(\theta_{n-1})] |\varphi(\theta_{n-1})|^{2}     \no \\
& \quad \times \bigg[\int_{0}^{\pi} d\theta_{n-1} \, |\varphi'(\theta_{n-1})|^{2}+[(n-2)(n-4)/4][\sin(\theta_{n-1})]^{-2} 
|\varphi(\theta_{n-1})|^{2}\bigg]^{-1}\Bigg\},      \no\\
&\hspace{9.5cm} n\geq 3.        \lb{3.41a}
\end{align} 

Employing the fact that
\begin{equation}
\bigg(- \f{d^2}{dx^2} + \f{s^2 - (1/4)}{\sin^2(x)}\bigg)\bigg|_{C_0^{\infty}((0,\pi))} \geq [(1/2) + s]^2 I_{L^2((0,\pi); dx)}, 
\quad s \geq 0,
\end{equation}
(this follows from \cite[Sect.~4]{GK85} for $s > 0$ and extends to $s=0$ utilizing \cite[Subsect.~6.1]{GLN20}) one concludes the following variant of Hardy's inequality (upon taking $s=0$) with optimal constants $1/4$,
\begin{equation}
\int_0^{\pi} dx \, |\varphi'(x)|^2 \geq \f{1}{4} \int_0^{\pi} dx \, \f{|\varphi(x)|^2}{\sin^2(x)} 
+ \f{1}{4} \int_0^{\pi} dx \, |\varphi(x)|^2, \quad \varphi \in C_0^{\infty}((0,\pi)), 
\end{equation}
which, by a density argument, extends to 
\begin{equation}
\int_0^{\pi} dx \, |\varphi'(x)|^2 \geq \f{1}{4} \int_0^{\pi} dx \, \f{|\varphi(x)|^2}{\sin^2(x)} 
+ \f{1}{4} \int_0^{\pi} dx \, |\varphi(x)|^2,  \quad \varphi \in H_0^1((0,\pi))   \lb{3.41b}
\end{equation}
(see \cite{GPS21}). 
Thus, employing \eqref{3.41b} in \eqref{3.41a} yields 
\begin{align}
\gamma_{c,n}^{-1} &\leq \underset{\varphi \in H_{0}^{1}((0,\pi))\backslash\{0\}}{\sup} \Bigg\{\int_{0}^{\pi} d\theta_{n-1} \, 
[- \cos(\theta_{n-1})] |\varphi(\theta_{n-1})|^{2}     \no \\
& \quad \times \bigg[\int_{0}^{\pi} d\theta_{n-1} \, \big\{(1/4)+ [(n-2)(n-4)/4]\big\}[\sin(\theta_{n-1})]^{-2} 
|\varphi(\theta_{n-1})|^{2}\bigg]^{-1}\Bigg\}       \no\\
&\leq \f{4}{(n-2)(n-4) + 1} \underset{\varphi \in H_{0}^{1}((0,\pi))\backslash\{0\}}{\sup} \Bigg\{\int_{0}^{\pi} [\sin(\theta_{n-1})]^{-2} d\theta_{n-1} \,   \no \\
& \hspace*{5.75cm}  \times 
[- \cos(\theta_{n-1})] [\sin(\theta_{n-1})]^2 |\varphi(\theta_{n-1})|^{2}     \no \\
& \hspace*{4cm}  \times \bigg[\int_{0}^{\pi} [\sin(\theta_{n-1})]^{-2} d\theta_{n-1} \, 
|\varphi(\theta_{n-1})|^{2}\bigg]^{-1}\Bigg\}       \lb{3.41ba} \\ 
&= \f{4}{(n-2)(n-4) + 1} \underset{\varphi \in H_{0}^{1}((0,\pi))\backslash\{0\}}{\sup} \Bigg\{\int_{\pi/2}^{\pi} [\sin(\theta_{n-1})]^{-2} d\theta_{n-1}    \no \\
& \hspace*{5.75cm}  \times  
[- \cos(\theta_{n-1})] [\sin(\theta_{n-1})]^2 |\varphi(\theta_{n-1})|^{2}     \no \\
& \hspace*{4.9cm}  \times \bigg[\int_{\pi/2}^{\pi} [\sin(\theta_{n-1})]^{-2} d\theta_{n-1} \, |\varphi(\theta_{n-1})|^{2}\bigg]^{-1}\Bigg\}       \no\\ 
& \leq \big[8\big/3^{3/2}\big] [(n-2)(n-4) + 1]^{-1}, \quad n\geq 4.         \lb{3.41c}
\end{align} 
Here we used the estimate, 
\begin{equation}
- \cos(\theta) \sin^2(\theta) \leq 2 \big/ 3^{3/2}, \quad \theta \in [\pi/2,\pi],
\end{equation}
and the fact that due to the sign change of $\cos(\theta)$ as $\theta$ crosses $\pi/2$, the numerator in \eqref{3.41ba} diminishes and the denominator in \eqref{3.41ba} increases, altogether diminishing the ratio in \eqref{3.41ba} 
if $\varphi(\dott)$ has support in $[0,\pi/2]$. Thus, one is justified assuming that $\varphi(\dott)$ has support in $[\pi/2,\pi]$ only. 

In the case $n=3$, the factor $[(n-2)(n-4) + 1]/4$ in ]\eqref{3.41ba} vanishes, and hence we now employ the additional 
term $\|\varphi\|^2_{L^2((0,\pi);dx)}/4$ in \eqref{3.41b} to arrive at 
\begin{align}
\gamma_{c,3}^{-1} &\leq \underset{\varphi \in H_{0}^{1}((0,\pi))\backslash\{0\}}{\sup} \Bigg\{\int_{0}^{\pi} d\theta_{2} \, 
[- \cos(\theta_{2})] |\varphi(\theta_{2})|^{2} \bigg[ \f{1}{4} \int_{0}^{\pi} d\theta_{2} \, 
|\varphi(\theta_{2})|^{2}\bigg]^{-1}\Bigg\}      \no\\
&\leq 4 \underset{\varphi \in H_{0}^{1}((0,\pi))\backslash\{0\}}{\sup} \Bigg\{\int_{\pi/2}^{\pi} d\theta_{2} \, 
[- \cos(\theta_{2})] |\varphi(\theta_{2})|^{2} \bigg/ \int_{\pi/2}^{\pi} d\theta_{2} \, |\varphi(\theta_{2})|^{2}\Bigg\}       \no\\
& \leq 4.         \lb{3.41d}
\end{align} 
Altogether, this implies the lower bound in \eqref{3.41e} and hence improves on Remark \ref{r3.2}\,$(ii)$ for $n \geq 5$. (For $n=4$ one can include the term $\|\varphi\|^2_{L^2((0,\pi));dx)}/4$ to improve the lower bound, but the actual details become so unwieldy that we refrain from doing so.) For $n=3,4$ we just recalled \eqref{3.2a}.

Next, introducing the functionals
\begin{align}
& F_n(\varphi) = \int_{0}^{\pi} d\theta_{n-1} \, 
[- \cos(\theta_{n-1})] |\varphi(\theta_{n-1})|^{2}     \no \\
& \quad \times \bigg[\int_{0}^{\pi} d\theta_{n-1} \, \big\{|\varphi'(\theta_{n-1})|^{2}+[(n-2)(n-4)/4][\sin(\theta_{n-1})]^{-2} 
|\varphi(\theta_{n-1})|^{2}\big\}\bigg]^{-1},     \no \\
& \hspace*{5.5cm} \varphi \in H_{0}^{1}((0,\pi))\backslash\{0\}, \; n \in \bbN, \; n \geq 3,
\end{align}
one concludes as in \eqref{3.41c} that
\begin{align}
\begin{split} 
F_n(\varphi) & \leq \f{\int_{\pi/2}^{\pi} [\sin(\theta_{n-1})]^{-2} d \theta_{n-1} \, [- \cos(\theta_{n-1})] \sin^2(\theta_{n-1}) 
|\varphi(\theta_{n-1})|^2 }{\int_{\pi/2}^{\pi} [\sin(\theta_{n-1})]^{-2} d \theta_{n-1} \, |\varphi(\theta_{n-1})|^2}     \\
& \leq 2 \big/ 3^{3/2}, \quad n \in \bbN, \; n \geq 3,
\end{split} 
\end{align}
is uniformly bounded with respect to $n$ and strictly monotonically decreasing with respect to $n$. Consequently, also 
\begin{equation}
\gamma_{c,n}^{-1} (n-2)(n-4)/4 =  \underset{\varphi \in H_{0}^{1}((0,\pi))\backslash\{0\}}{\sup} F_n(\varphi)
\end{equation}
is bounded and monotonically decreasing with respect to $n$ and hence has a limit as $n \to \infty$, proving \eqref{3.38}.

Finally, to prove the upper bound in \eqref{3.41e} one can argue as follows. Introducing 
$\varphi_0 \in H^1_0((0,\pi))\backslash \{0\}$ via
\begin{equation}
\varphi_0(\theta) = \begin{cases}
0, & \theta \in [0, \pi/2], \\
\sin(2 \theta), & \theta \in [\pi/2,\pi],
\end{cases} 
\end{equation}
then,
\begin{align}
& \int_{\pi/2}^{\pi} d\theta \, [- \cos(\theta)] \sin^2(2 \theta) = 8/15,    \no \\
& \int_{\pi/2}^{\pi} d\theta \, \big\{4 \cos^2(2 \theta) + [(n-2)(n-4)/4] [\sin(\theta)]^{-2} 4 \sin^2(\theta) \cos^2(\theta)\big\}     \\
& \quad = \int_{\pi/2}^{\pi} d\theta \, \big[4 \cos^2(2 \theta) + (n-2)(n-4) \cos^2(\theta)\big] = \pi [4 + (n-2)(n-4)]/4,   \no
\end{align}
and hence (cf.\ \eqref{3.41a})
\begin{equation}
\gamma_{c,n}^{-1} \geq \f{32}{15 \pi [(n-2)(n-4) + 4]}, \quad n \geq 3, 
\end{equation}
completes the proof of \eqref{3.41e}. 
\end{proof}

\begin{remark} \lb{r3.3a} 
$(i)$ Since $H^1_0((0,\pi))$ embeds compactly into $L^2((0,\pi); d\theta)$, the supremum in \eqref{3.41a} (unlike that in \eqref{3.40a}) is actually attained, that is, for a particular $\varphi_n \in H^1_0((0,\pi)) \backslash\{0\}$, 
\begin{align}
\gamma_{c,n}^{-1} &= \int_{0}^{\pi} d\theta_{n-1} \, 
[- \cos(\theta_{n-1})] |\varphi_n(\theta_{n-1})|^{2}     \no \\
& \quad \times \bigg[\int_{0}^{\pi} d\theta_{n-1} \, |\varphi'(\theta_{n-1})|^{2}+[(n-2)(n-4)/4][\sin(\theta_{n-1})]^{-2} 
|\varphi_n(\theta_{n-1})|^{2}\bigg]^{-1}\Bigg\},      \no\\
&\hspace{9cm} \quad  n\geq 3.        \lb{3.57}
\end{align} 
However, since the $n$-dependence of $\varphi_n$ appears to be beyond our control, computing the exact value of $C_0$ in \eqref{3.38} remains elusive. \\[1mm]
$(ii)$ The differential equation underlying \eqref{3.41a} is of the type 
\begin{equation}
- y''(\theta) + [(n-2)(n-4)/4] [\sin(\theta)]^{-2} y(\theta) = - \gamma_{c,n} \cos(\theta) y(\theta), \quad \theta \in (0,\pi), 
\end{equation}
which naturally leads to the Birman--Schwinger-type eigenvalue problem 
\begin{align}
\Big(h_n^{-1/2} [- \cos(\theta)] h_n^{-1/2} v\Big)(\theta) 
= \lambda_n v(\theta), \quad v = h_n^{1/2} y, 
\end{align}
where $h_n$ denotes the Friedrichs extension of the preminimal operator $\dot h_{n, min}$ in $L^2((0,\pi); d\theta)$ defined by 
\begin{equation}
\big(\dot h_{n,min} g\big)(\theta) = - g''(\theta) + [(n-2)(n-4)/4] [\sin(\theta)]^{-2} g(\theta), \quad g \in C_0^{\infty}((0,\pi)).
\end{equation}
One observes that $\dot h_{n,min}$ is essentially self-adjoint for $n \geq 5$ and hence boundary conditions at $\theta = 0, \pi$, familiar for singular second-order differential operators of Bessel-type (see \cite[Subsection~6.1]{GLN20}), are only required for $n=3,4$. The Birman--Schwinger operator 
\begin{equation}
T_n = h_n^{-1/2} [- \cos(\theta)] h_n^{-1/2} 
\end{equation} 
in $L^2((0,\pi); d\theta)$ is compact (in fact, Hilbert--Schmidt) upon inspecting its integral kernel and hence by the 
Raleigh--Ritz quotient in \eqref{3.41a}, $\gamma_{c,n}^{-1}$ is the largest eigenvalue for $T_n$. Finally, introducing the unitary operator 
\begin{equation}
(U f)(\theta) = f(\pi - \theta), \quad \theta \in (0,\pi), \; f \in L^2((0,\pi); d\theta), 
\end{equation}
in $L^2((0,\pi); d\theta)$, one verifies that 
\begin{equation}
U T_n U^{-1} = - T_n, 
\end{equation}
and hence the spectrum of $T_n$ is symmetric with respect to the origin. 
\hfill $\diamond$
\end{remark} 
It is well-known that Hardy's inequality \eqref{1.1} is strict, that is, equality holds in \eqref{1.1} for some $f \in D^1(\bbR^n)$ if and only if $f=0$. More general results regarding strictness for weighted Hardy--Sobolev or Caffarelli--Kohn--Nirenberg inequalities based on variational techniques can be found, for instance, in \cite{CW01}, \cite{CC93}. Strictness in the case of the Hardy inequality was discussed in \cite{VZ00}. Thus, we next turn to strictness of inequality \eqref{3.2} on $H^1(\bbR^n)$ employing a quadratic form approach.

To set the stage we briefly recall a few facts on quadratic forms generated by symmetric operators $A$ bounded from below and the associated Friedrichs extension (to be denoted by $A_F$) of $A$. 

Let $A$ be a densely defined symmetric operator in the Hilbert space $\cH$ bounded from below, that is, 
$A \subseteq A^*$ and for some $c \in \bbR$, $A \ge c I_{\cH}$. Without loss of generality we put $c=0$ in the following. We denote by $\ol A$ the closure of $A$ 
in $\cH$, and introduce the associated forms in $\cH$, 
\begin{align}
& q_A(f,g) = (f, A g)_{\cH}, \quad f, g \in \dom(q_A) = \dom(A),   \lb{3.39} \\
& q_{\ol A}(f,g) = (f, {\ol A} g)_{\cH}, \quad f, g \in \dom(q_{\ol A}) = \dom({\ol A}),    \lb{3.40} 
\end{align} 
then the closures of $q_A$ and $q_{\ol A}$ coincide in $\cH$ (cf., e.g., \cite[Lemma~5.1.12]{BHS20}) 
\begin{equation}
\ol{q_A} = \ol{q_{\ol A}}    \lb{3.41}
\end{equation}
and the first representation theorem for forms (see, e.g., \cite[Theorem~4.2.4]{EE18}, 
\cite[Theorem~VI.2.1, Sect.~VI.2.3]{Ka95}) yields 
\begin{equation}
\ol{q_A} (f,g) = (f, A_F g)_{\cH}, \quad f \in \dom (\ol{q_A}), \; g \in \dom(A_F),     \lb{3.42}
\end{equation}
where $A_F \geq 0$ represents the self-adjoint Friedrichs extension of $A$. Due to the fact 
\eqref{3.41}, one infers (cf., e.g., \cite[Lemma~5.3.1]{BHS20})
\begin{equation}
A_F = (\ol A)_F.    \lb{3.43} 
\end{equation} 
The second representation theorem for forms (see, e.g., \cite[Theorem~4.2.8]{EE18}, 
\cite[Theorem~VI.2.123]{Ka95}) then yields the additional result
\begin{equation}
\ol{q_A} (f,g) = \big(A_F^{1/2} f, A_F^{1/2} g\big)_{\cH}, \quad 
 f, g \in \dom(\ol{q_A}) = \dom\big(A_F^{1/2}\big).
\end{equation} 
Moreover, one has the fact (see, e.g., \cite[Theorem~5.3.3]{BHS20}, \cite[Corollary~4.2.7]{EE18}, 
\cite[Theorem~10.17]{Sc12})
\begin{equation}
\dom(A_F) = \dom(\ol{q_A}) \cap \dom(A^*) = \dom\big(A_F^{1/2}\big) \cap \dom(A^*).   \lb{5.45}
\end{equation}

\begin{theorem} \lb{t3.5} 
Assume Hypothesis \ref{h2.2}. Then inequality \eqref{3.2} is strict on $H^1(\bbR^n)$, $n \geq 3$, that is, equality holds in \eqref{3.2} for some $f \in H^1(\bbR^n)$ if and only if $f = 0$. 
\end{theorem}
\begin{proof}
We first discuss the simpler case $\gamma \in [0,\gamma_{c,n})$. In this case the inequality \eqref{3.2} implies that the sesquilinear form $\gamma \, q_u$, $u \in \bbR^n$, $|u| = 1$, where 
\begin{equation}
q_u (f,g) = \int_{\bbR^n} d^n x \, (u,x) |x|^{-3} \ol{f(x)} g(x), 
\quad f,g \in \dom(q_u) = H^1(\bbR^n),
\end{equation}
is bounded relative to the form $Q_{H_0}$ of the Laplacian $H_0 = - \Delta$ on $\dom(H_0) = H^2(\bbR^n)$, 
\begin{align}
\begin{split} 
Q_{H_0}(f,g) = (\nabla f, \nabla g)_{[L^2(\bbR^n)]^n} = \big(H_0^{1/2}f, H_0^{1/2} g\big)_{L^2(\bbR^n)},& \\ 
 f,g  \in \dom(Q_{H_0}) = H^1(\bbR^n),&
 \end{split} 
\end{align}
with relative bound strictly less than one. Hence the form
\begin{equation}
Q_{\gamma}(f,g) = Q_{H_0}(f,g) + \gamma \, q_u (f,g), \quad f,g \in \dom(Q_{\gamma}) = H^1(\bbR^n),
\; \gamma \in [0,\gamma_{c,n}), 
\end{equation}
 is densely defined, nonnegative, and closed. Moreover, since $C_0^{\infty}(\bbR^n \backslash \{0\})$ is dense in $H^1(\bbR^n)$ for $n \geq 2$ (cf., e.g., \cite[p.~33--35]{Fa75}), that is, $C_0^{\infty}(\bbR^n \backslash \{0\})$ is a core for $Q_{H_0}$ (equivalently, a form core for $H_0$), and hence also a core for $Q_{\gamma}$, 
\begin{equation}
\ol{Q_{\gamma}|_{C_0^{\infty}(\bbR^n \backslash \{0\})}} = Q_{\gamma}, \quad \gamma \in [0,\gamma_{c,n}).  
\end{equation} 
Thus, the self-adjoint, nonnegative operator $H_{Q_{\gamma}}$, uniquely associated with $Q_{\gamma}$ by the first representation theorem for forms coincides with the Friedrichs extension of the minimal operator associated with the differential expression $L_{\gamma}$ in \eqref{2.18},
\begin{equation}
H_{\gamma,min} = (- \Delta + V_{\gamma})|_{C_0^{\infty}(\bbR^n\backslash\{0\})} \geq 0, \quad \gamma \in [0,\gamma_{c,n}), 
\end{equation}
that is,
\begin{equation}
H_{Q_{\gamma}} = (H_{\gamma,min})_F \geq 0, \quad \gamma \in [0,\gamma_{c,n}). 
\end{equation}
In turn, since $H_{\gamma}$ coincides with the direct sum of Friedrichs extensions in \eqref{2.34}, one concludes that $H_{Q_{\gamma}}$ coincides with $H_{\gamma}$, and hence, 
\begin{equation}
H_{Q_{\gamma}} = H_{\gamma} = \big((- \Delta + V_{\gamma})|_{C_0^{\infty}(\bbR^n\backslash\{0\})}\big)_F,  
\quad \gamma \in [0,\gamma_{c,n}), 
\end{equation}
in particular, 
\begin{equation}
Q_{\gamma}(f,g) = \big(H_{\gamma}^{1/2} f, H_{\gamma}^{1/2} g\big)_{L^2(\bbR^n)}, \quad 
f,g \in \dom(Q_{\gamma}) = H^1(\bbR^n), \; \gamma \in [0,\gamma_{c,n}). 
\end{equation} 
Thus, equality in the inequality \eqref{3.2} for some $f_0 \in H^1(\bbR^n)$ implies 
\begin{align} 
0 &= \int_{\bbR^n} d^n x \, |(\nabla f_0)(x)|^2 + \gamma \int_{\bbR^n} d^n x \, (u,x) |x|^{-3} |f_0(x)|^2    \\
& = Q_{H_0}(f_0,f_0) + \gamma \, q_u (f_0,f_0) = Q_{\gamma}(f_0,f_0) = \big\|H_{\gamma}^{1/2} f_0 \big\|_{L^2(\bbR^n)}^2, \quad \gamma \in [0,\gamma_{c,n}),   \no 
\end{align}
and hence,
\begin{equation}
f_0 \in \ker\big(H_{\gamma}^{1/2}\big) = \ker (H_{\gamma}) = \{0\}, \quad \gamma \in [0,\gamma_{c,n}),  
\end{equation}
by \eqref{3.2b}. 

The case $\gamma = \gamma_{c,n}$ follows analogous lines but is a bit more involved as now the form 
$\gamma_{c,n} \, q_u$ is bounded relative to the form $Q_{H_0}$ with relative bound equal to one. 

Since by inequality \eqref{3.2} 
\begin{equation}
H_{\gamma_{c,n},min} = (- \Delta + V_{\gamma_{c,n}})|_{C_0^{\infty}(\bbR^n\backslash\{0\})} \geq 0,
\end{equation}
the form
\begin{align}
\dot Q_{\gamma_{c,n}}(f,g) &= (f, H_{\gamma_{c,n},min} g)_{L^2(\bbR^n)}    \\
&= Q_{H_0}(f,g) + \gamma_{c,n} q_u (f,g), \quad f,g \in \dom\big(\dot Q_{\gamma_{c,n}}\big) = C_0^{\infty}(\bbR^n \backslash \{0\}),   \no 
\end{align}
is closable and we denote its closure in $L^2(\bbR^n)$ by $Q_{\gamma_{c,n}}$. Thus,  
$Q_{\gamma_{c,n}} \geq 0$ and hence the self-adjoint, nonnegative operator $H_{Q_{\gamma_{c,n}}}$ uniquely associated with $Q_{\gamma_{c,n}}$ in $L^2(\bbR^n)$ is the Friedrichs extension of 
$H_{\gamma_{c,n},min}$,  
\begin{equation}
H_{Q_{\gamma_{c,n}}} = (H_{\gamma_{c,n},min})_F \geq 0.
\end{equation}
Again, by the second representation theorem 
\begin{equation}
Q_{\gamma_{c,n}} (f,g) = 
\big(H_{Q_{\gamma_{c,n}}}^{1/2} f, H_{Q_{\gamma_{c,n}}}^{1/2} g\big)_{L^2(\bbR^n)}, 
\quad f,g \in \dom(Q_{\gamma_{c,n}}) = \dom\big(H_{Q_{\gamma_{c,n}}}^{1/2}\big). 
\end{equation}
However, unlike in the case $\gamma \in [0,\gamma_{c,n})$, since the form 
$\gamma_{c,n} \, q_u$ is bounded relative to the form $Q_{H_0}$ with relative bound equal to one, one now has possible cancellations between the forms $Q_{H_0}$ and $\gamma_{c,n} \, q_u$ and hence concludes that 
\begin{equation}
 \dom(Q_{\gamma_{c,n}}) = \dom\big(H_{Q_{\gamma_{c,n}}}^{1/2}\big) \supseteq H^1(\bbR^n)
\end{equation}
(see also Remark \ref{r3.4}). The rest of the proof now follows the case $\gamma \in [0,\gamma_{c,n})$ line by line, in particular,
\begin{equation}
H_{Q_{\gamma_{c,n}}} = H_{\gamma_{c,n}} 
= \big((- \Delta + V_{\gamma_{c,n}})|_{C_0^{\infty}(\bbR^n\backslash\{0\})}\big)_F,  
\end{equation}
and 
\begin{equation}
\ker\big(H_{\gamma_{c,n}}^{1/2}\big) = \ker (H_{\gamma_{c,n}}) = \{0\},  
\end{equation} 
again by \eqref{3.2b}, then proves strictness of \eqref{3.2} also for $\gamma = \gamma_{c,n}$.   
\end{proof}

\begin{remark} \lb{r3.4} 
We briefly illustrate the possibility of cancellations in $H_{\gamma_{c,n}}$. Let 
$\psi_{\gamma,0} = \psi_{\gamma,0}(\theta_{n-1})$, $\gamma \geq 0$, be the unique (up to constant 
multiples) eigenfunction of 
$\Lambda_{\gamma,n}$ corresponding to its lowest eigenvalue $\lambda_{\gamma,n,0}$, see \eqref{3.14}, 
and introduce 
\begin{equation}
\Psi_{\gamma,0} (x) = |x|^{-(n-2)/2} \psi_{\gamma,0}(\theta_{n-1}), \quad x \in \bbR^n \backslash \{0\}, 
\; \gamma \geq 0. 
\end{equation} 
Then an elementary computation reveals that 
\begin{equation}
L_{\gamma} \Psi_{\gamma,0} = \big\{\big[(n-2)^2/4\big] + \lambda_{\gamma,n,0}\big\} 
|x|^{- (n+2)/2} \psi_{\gamma,0}(\theta_{n-1}), 
\end{equation}
in the sense of distributions. In particular, if $\gamma = \gamma_{c,n}$, and hence, 
$\lambda_{\gamma_{c,n},n,0} = - (n-2)^2/4$, one obtains 
\begin{equation}
L_{\gamma_{c,n}} \Psi_{\gamma_{c,n},0} = 0,
\end{equation} 
in the distributional sense. Thus, introducing 
\begin{equation}
f_0(x) = |x|^{-(n-2)/2} \psi_{\gamma_{c,n},0}(\theta_{n-1}) \phi(|x|), \quad x \in \bbR^n \backslash \{0\}, 
\end{equation}
one concludes that 
\begin{equation}
f_0 \in \dom(H_{\gamma_{c,n}}) \subset \dom\big(H_{\gamma_{c,n}}^{1/2}\big).  
\end{equation}
However, since $(\partial f_0/\partial r) \notin L^2(\bbR^n)$, the elementary fact  
\begin{equation}
\bigg|\bigg(\f{\partial f_0}{\partial r}\bigg)(x)\bigg| = \bigg|\f{x}{|x|} \cdot (\nabla f_0)(x)\bigg| 
\leq |(\nabla f_0)(x)|
\end{equation} 
implies that 
\begin{equation}
f_0 \notin H^1(\bbR^n), \; (u,x) |x|^{-3} |f_0|^2 \notin L^1(\bbR^n), 
\end{equation}
illustrating possible cancellations between $H_0$ and $\gamma_{c,n} (u,x) |x|^{-3}$. 
\hfill $\diamond$ 
\end{remark}

\begin{remark} \lb{r3.5} 
Next, we briefly discuss the remaining case $n=2$. In this situation, the Laplace--Beltrami operator 
$- \Delta_{\bbS^1}$ in $L^2\big(\bbS^1\big)$ can be characterized by 
\begin{align}
& (- \Delta_{\bbS^1} f)(\theta_1) = - f''(\theta_1),    \quad \theta_1 \in (0,2\pi),   \no \\
& f \in \dom(- \Delta_{\bbS^1}) = \big\{g \in L^2((0,2\pi);d\theta_1) \, \big| \, g, g' \in AC([0,2\pi]); \\ 
& \hspace*{3.1cm} g(0) = g(2 \pi), \, g'(0) = g'(2 \pi); \, g'' \in L^2((0,2\pi);d\theta_1)\big\},   \no 
\end{align}
with
\begin{align}
& \; \sigma(- \Delta_{\bbS^1}) = \big\{\ell^2\big\}_{\ell \in \bbN_0},    \\
& - \Delta_{\bbS^1} e^{\pm i \ell \theta_1} = \ell^2 e^{\pm i \ell \theta_1}, \quad \theta_1 \in (0,2\pi), \; 
\ell \in \bbN_0.
\end{align}
The resulting Mathieu operator $\Lambda_{\gamma,2}$ in $L^2((0,2\pi); d\theta_1)$ (cf.\ \eqref{2.20}), of the form
\begin{equation}
\Lambda_{\gamma,2} = -\f{d^2}{d \theta_1^2} + \gamma\cos(\theta_{1}), \quad 
\dom(\Lambda_{\gamma,2}) = \dom(-\Delta_{\bbS^1}),
\end{equation}
has extensively been studied in the literature, see, for instance \cite[Ch.~2]{MS54}. More generally, the least periodic eigenvalue of Hill operators (i.e., situations where $\cos(\theta_1)$ is replaced by a $2\pi$-periodic, locally integrable potential $q(\theta_1)$) has received enormous attention, see for instance, 
\cite{Bl63}, \cite{GGS92}, \cite{Ka52}, \cite{Mo56/57}, \cite{Pu51}, \cite{RS16}, \cite{St79}, \cite{Un61}, and \cite{Wi51}. Applied to the Mathieu operator $\Lambda_{\gamma,2}$ at hand, the results obtained 
(cf.\ the discussion in \cite{GGS92}) imply, 
\begin{align}
\begin{split}
&  - \gamma^2\big/\big[8 \pi^2\big] < \lambda_{\gamma,2,0} < 0, \quad \gamma \in (0,\infty), \lb{3.25}  \\
& \quad \text{there exists $c_0 \in (0,\infty)$ such that } \, \lambda_{\gamma,2,0} \leq - c_0 \gamma^2, 
\quad \gamma \in [0,1]. 
\end{split} 
\end{align}
In particular, this proves the absence of a critical coupling constant $0 < \gamma_{c,2}$ for $n=2$ (equivalently, the critical constant in two dimensions equals zero, $\gamma_{c,2}=0$), explaining why we had to limit ourselves to $n \geq 3$ in the bulk of this paper. \hfill $\diamond$
\end{remark}

\begin{remark} \lb{r3.6} 
While thus far we focused primarily on lower semiboundedness of $H_{\gamma}$, the direct sum considerations in Section \ref{s2} equally apply to essential self-adjointness of 
$H_{\gamma}|_{C_0^{\infty}(\bbR^n \backslash \{0\})}$. Indeed, returning to the operator in \eqref{2.27}, one notes that 
\begin{equation}
\bigg[- \f{d^2}{dr^2} + \f{c}{r^2}\bigg]\bigg|_{C_0^{\infty}((0,\infty))} \, \text{ is essentially self-adjoint 
if and only if $c \geq 3/4$.}    \lb{3.51}
\end{equation}
The criterion \eqref{3.51} combined with \eqref{2.34}, \eqref{2.35} thus implies that 
\begin{align}
\begin{split} 
& H_{\gamma}|_{C_0^{\infty}(\bbR^n \backslash \{0\})} \, \text{ is essentially self-adjoint} \\ 
& \quad \text{if and only if } \, \lambda_{\gamma,n,0} \geq - n(n-4)/4.     \lb{3.52}
\end{split}
\end{align} 
${}$ \hfill $\diamond$ 
\end{remark}

\section{A Numerical Approach} \lb{s4}

Having verified the existence and uniqueness of critical dipole moments $\gamma_{c,n}$ for all dimensions $n\geq 3$, and having shown some of the properties of $\lambda_{\gamma,n,0}$, this section is devoted to a description of a numerical method for computing $\gamma_{c,n}$, in analogy to the Legendre expansion in \cite{CG07}.  

To set up the numerical algorithm one can argue as follows: Given \eqref{3.14}, we are interested in solving this eigenvalue problem in the particular scenario where $\gamma$ ranges from $0$ to 
$\gamma_{c,n}$, observing that $\lambda_{\gamma_{c,n},n,0} = - (n-2)^2/4$ (cf.\ \eqref{2.31}). Restricting 
$- \Delta_{\bbS^{n-1}}$ in \eqref{A.16} to $\cL^{n}$ as in \eqref{3.3}--\eqref{3.5}, \eqref{3.14} reduces to solving the eigenvalue problem associated with $\Lambda_{\gamma,\cL^{n}}$ in 
$L^2\big((0,\pi); [\sin(\theta_{n-1})]^{n-2} d\theta_{n-1}\big)$ of the type,  
\begin{align}
\begin{split} 
& -\f{d^{2}\Psi_{\gamma} (\theta_{n-1})}{d\theta_{n-1}^{2}}-(n-2)\cot(\theta_{n-1})\f{d\Psi_{\gamma}(\theta_{n-1})}{d\theta_{n-1}} +
\gamma \cos(\theta_{n-1}) \Psi_{\gamma} (\theta_{n-1})    \\
& \quad = \lambda_{\gamma,n,0} \Psi_{\gamma}(\theta_{n-1}), \quad \gamma \in (0,\gamma_{c,n}], \; 
\theta_{n-1} \in (0,\pi). 
	\lb{4.1}
\end{split} 
\end{align}
Expanding $\Psi_{\gamma} (\, \cdot \,)$ in normalized Gegenbauer polynomials, one obtains
\begin{align}
& \Psi_{\gamma} (\theta_{n-1}) =\sum\limits_{\ell=0}^{\infty} d_{\ell}(\gamma) \bigg[\f{\ell !(2\ell+n-2)}{2^{4-n}\pi\Gamma(\ell+n-2)}\bigg]^{1/2}\Gamma((n-2)/2) C_{\ell}^{(n-2)/2}(\cos(\theta_{n-1})),   \no \\  
& \hspace*{9cm} \theta_{n-1} \in (0,\pi), \lb{4.2}
\end{align}
where $d_{\ell}(\gamma)$ are appropriate expansion coefficients.  Since the Gegenbauer polynomial 
$C_{\ell}^{(n-2)/2}(\cos(\dott))$ is an eigenfunction of $- \Delta_{\bbS^{n-1}}$ corresponding to the eigenvalue $\ell(\ell+n-2)$, \eqref{4.1} becomes
\begin{align}
\begin{split} 
& \sum\limits_{\ell=0}^{\infty} [\ell(\ell+n-2)+\gamma \cos(\theta_{n-1}) 
- \lambda_{\gamma,n,0}] \, d_{\ell}(\gamma) \bigg[\f{\ell !(2\ell+n-2)}{2^{4-n}\pi\Gamma(\ell+n-2)}\bigg]^{1/2} \\
& \hspace*{5.5mm} \times \Gamma((n-2)/2) C_{\ell}^{(n-2)/2}(\cos(\theta_{n-1})) = 0. 
	\lb{4.3}
	\end{split} 
\end{align}
Next, we will exploit the following recurrence relation of Gegenbauer polynomials, 
\begin{align}
& \cos(\theta_{n-1})C_{\ell}^{(n-2)/2}(\cos(\theta_{n-1}))   \no \\
& \quad =\f{\ell+1}{2\ell+n-2}\left(C_{\ell+1}^{(n-2)/2}(\cos(\theta_{n-1}))+\f{\ell+n-3}{\ell+1}
C_{\ell-1}^{(n-2)/2}(\cos(\theta_{n-1}))\right),   \no\\
& \hspace*{9.6cm} \ell \in\bbN_{0} 
	\lb{4.4}  
\end{align}
(with $C_{-1}^{(n-2)/2}(\,\cdot\,)\equiv 0$) to expand the term $\gamma \cos(\theta_{n-1})$. For the $(\ell-1)$-term, one infers 
\begin{align}
& \gamma \cos(\theta_{n-1})d_{\ell-1}(\gamma) 
\bigg[\f{(\ell-1)!(2\ell+n-4))^2}{2^{4-n}\pi\Gamma(\ell+n-3)}\bigg]^{1/2}
\Gamma((n-2)/2) C_{\ell-1}^{(n-2)/2}(\cos(\theta_{n-1}))    \no \\
& \quad =\gamma d_{\ell-1}(\gamma) \bigg[\f{(\ell-1)!(2\ell+n-4)}{2^{4-n}\pi\Gamma(\ell+n-3)}\bigg]^{1/2} 
\f{\ell\, \Gamma((n-2)/2)}{2\ell+n-4} C_{\ell}^{(n-2)/2}(\cos(\theta_{n-1})),  \no \\
& \hspace*{9.9cm} \ell \in\bbN_{0}
 	\lb{4.5}
\end{align}
(where $d_{-1}(\gamma) = 0$) and for the $(\ell+1)$-term, one obtains 
\begin{align}
& \gamma \cos(\theta_{n-1})d_{\ell+1}(\gamma) \bigg[\f{(\ell+1)!(2\ell+n)}{2^{4-n}\pi\Gamma(\ell+n-1)}\bigg]^{1/2} \Gamma((n-2)/2) C_{\ell+1}^{(n-2)/2}(\cos(\theta_{n-1}))     \no\\
& \quad =\gamma d_{\ell+1}(\gamma) \bigg[\f{(\ell+1)!(2\ell+n)}{2^{4-n}\pi\Gamma(\ell+n-1)}\bigg]^{1/2}
\Gamma((n-2)/2) \f{\ell+n-2}{2\ell+n} C_{\ell}^{(n-2)/2}(\cos(\theta_{n-1})), \no \\
& \hspace*{10cm} \ell \in\bbN_{0}.
 	\lb{4.6}
\end{align}
The $\ell$-term maintains its form 
\begin{align}
& [\ell(\ell+n-2)-\lambda_{\gamma,n,0}] \, 
d_{\ell}(\gamma) \bigg[\f{\ell!(2\ell+n-2)}{2^{4-n}\pi\Gamma(\ell+n-2)}\bigg]^{1/2}
\Gamma((n-2)/2)   \no \\
& \quad \times C_{\ell}^{(n-2)/2}(\cos(\theta_{n-1})), \quad \ell \in\bbN_{0},
	\lb{4.7}
\end{align}
so one can divide all terms by the normalizing factor from \eqref{4.2} (since the orthogonality of the Gegenbauer polynomials mandates that every term under the sum in \eqref{4.3} individually vanishes), obtaining
\begin{align}
& \sum\limits_{\ell=0}^{\infty}\bigg\{[\ell(\ell+n-2)-\lambda_{\gamma,n,0}] d_{\ell}(\gamma)    \no \\
& \hspace*{7mm} 
+ \gamma \bigg(\bigg[\f{(2\ell+n-4)(\ell+n-3)}{\ell(2\ell+n-2)}\bigg]^{1/2}\,\f{\ell}{2\ell+n-4}d_{\ell-1}(\gamma)    	\lb{4.8} \\
& \hspace*{7mm} 
+\bigg[\f{(\ell+1)(2\ell+n)}{(2\ell+n-2)(\ell+n-2)}\bigg]^{1/2} \f{\ell+n-2}{2\ell+n}\,d_{\ell+1}(\gamma)\bigg)\bigg\}
C_{\ell}^{(n-2)/2}(\cos(\theta_{n-1}))=0.    \no 
\end{align}
Setting each coefficient equal to zero results in 
\begin{align}
\begin{split} 
& [\ell(\ell+n-2) - \lambda_{\gamma,n,0}] d_{\ell}(\gamma) 
+ \gamma \bigg(\bigg[\f{\ell(\ell+n-3)}{(2\ell+n-4)(2\ell+n-2)}\bigg]^{1/2}d_{\ell-1}(\gamma)    \\
& \quad + \bigg[\f{(\ell+1)(\ell+n-2)}{(2\ell+n-2)(2\ell+n)}\bigg]^{1/2} d_{\ell+1}(\gamma)\bigg)=0, \quad 
 \ell \in \bbN_{0}, 	\lb{4.9}
 \end{split} 
\end{align}
which one can rewrite as
\begin{align}
&\bigg[\f{(\ell+1)(\ell+n-2)}{(2\ell+n-2)(2\ell+n)}\bigg]^{1/2}d_{\ell+1}(\gamma)  
+ \bigg[\f{\ell(\ell+n-3)}{(2\ell+n-4)(2\ell+n-2)}\bigg]^{1/2} d_{\ell-1}(\gamma)     \no\\
&\quad = - \f{1}{\gamma} [\ell(\ell+n-2) - \lambda_{\gamma,n,0}] \,d_{\ell}(\gamma), 
\quad \gamma \in (0, \gamma_{c,n}], \; \ell \in \bbN_{0}. 
	\lb{4.10}
\end{align}
Equation \eqref{4.10} can be expressed as the generalized Jacobi operator eigenvalue problem in $\ell^{2}(\bbN_{0};w)$, 
\begin{align}
Jd(\gamma)=-\f{1}{\gamma}w(\gamma)d(\gamma), \quad \gamma \in (0, \gamma_{c,n}],    \lb{4.11}
\end{align}
where
\begin{align}
Jd(\gamma)=\begin{cases}a_{\ell+1}d_{\ell+1}(\gamma) + a_{\ell}d_{\ell-1}(\gamma), 
&\ell\in\bbN,\\a_{1}d_{1}(\gamma), &\ell=0,\end{cases} \quad 
w(\gamma)d(\gamma)= \big(w_{\ell}(\gamma) d_{\ell}(\gamma)\big)_{\ell\in\bbN_{0}},
	\lb{4.12}
\end{align}
and 
\begin{align}
\begin{split} 
& a_{\ell}=\bigg[\f{\ell(\ell+n-3)}{(2\ell+n-4)(2\ell+n-2)}\bigg]^{1/2}, 
\quad w_{\ell}(\gamma) = \ell(\ell+n-2) - \lambda_{\gamma,n,0},  \quad \ell \in \bbN_{0}.    \lb{4.13} 
\end{split} 
\end{align}
(One observes that $w_0(\gamma) \underset{\gamma \downarrow 0}{\longrightarrow} 0$ by \eqref{2.41}.) 

Explicitly, \eqref{4.12} yields the self-adjoint Jacobi operator $J$ in $\ell^2(\bbN_0;w)$ represented as a semi-infinite matrix with respect to the standard Kronecker-$\delta$ basis  
\begin{align}
Jd(\gamma)&=\begin{pmatrix}
0&a_{1}&0&\ldots&&&\\
a_{1}&0&a_{2}&0&\ldots&&\\
0&a_{2}&0&a_{3}&0&\ldots&\\
\vdots&0&a_{3}&0&a_{4}&0&\ldots\\
&\vdots&0&\ddots&\ddots&\ddots&\\
\end{pmatrix}
\begin{pmatrix}
d_{0}(\gamma) \\
d_{1}(\gamma) \\
d_{2}(\gamma) \\
\vdots\\
\phantom{\vdots}\\
\end{pmatrix}      \no \\
&= -\f{1}{\gamma}\begin{pmatrix}
w_{0}(\gamma) & 0& \ldots &&&\\
0& w_{1}(\gamma) & 0 & \ldots &&\\
\vdots& 0 & w_{2}(\gamma) & 0 & \ldots &\\
& \vdots & 0 & w_{3}(\gamma) & 0 & \ldots \\
& & \vdots & \ddots & \ddots & \ddots\\
\end{pmatrix}
\begin{pmatrix}
d_{0}(\gamma) \\
d_{1}(\gamma) \\
d_{2}(\gamma) \\
\vdots\\
\phantom{\vdots}\vspace{1mm}\\
\end{pmatrix}    \no \\
&=-\f{1}{\gamma}w(\gamma) d(\gamma), \quad \gamma \in (0,\gamma_{c,n}]. 
	\lb{4.14}
\end{align}

One would like to calculate $\gamma_{c,n}$ approximately using finite truncations of the matrix representation of $J$ in the first line of \eqref{4.14} - the feasibility of truncations will be made precise below.  In order for these approximants to converge, a transformation to a compact Jacobi operator becomes necessary. For this purpose one introduces the operator
\begin{align}
V(\gamma) =
\begin{cases}
\ell^{2}(\bbN_{0};w(\gamma))\to\ell^{2}(\bbN_{0}),\\
b \mapsto w^{1/2}(\gamma) b.
\end{cases}
	\lb{4.15}
\end{align}
That $V(\gamma)$ is unitary may be seen from
\begin{align}
||V(\gamma) b||_{\ell^{2}(\bbN_{0})}=||w^{1/2}(\gamma) b||_{\ell^{2}(\bbN_{0})}
=||b||_{\ell^{2}(\bbN_{0};w(\gamma))},\quad b\in\ell^{2}(\bbN_{0};w(\gamma)),
	\lb{4.16}
\end{align}
and the fact that $V(\gamma)$ is surjective and defined on all of $\ell^{2}(\bbN_{0};w)$.  Next, one transforms \eqref{4.14} into
\begin{equation}
V(\gamma)^{-1} J V(\gamma)^{-1} V(\gamma) d = - \f{1}{\gamma}V(\gamma)d(\gamma),   
\quad d(\gamma)\in\ell^{2}(\bbN_{0};w(\gamma)),
	\lb{4.17}
\end{equation}
which can equivalently be expressed as
\begin{equation}
w(\gamma)^{-1/2} J w(\gamma)^{-1/2} \wti d(\gamma) 
= - \f{1}{\gamma} \wti d(\gamma), \quad \wti d(\gamma) = V(\gamma) d(\gamma) \in\ell^{2}(\bbN_{0}).
	\lb{4.18}
\end{equation}
Introducing 
\begin{equation} 
K(\gamma)=w(\gamma)^{-1/2} J w(\gamma)^{-1/2} \in \cB\big(\ell^2(\bbN_0)\big), 
\quad \gamma \in (0,\gamma_{c,n}],     \lb{4.18a}
\end{equation}  
one can write \eqref{4.18} in the form
\begin{align}
\begin{split} 
K(\gamma) \wti d(\gamma) &= \begin{pmatrix}
0 &\wti a_{1}(\gamma) & 0 & \ldots & & & \\
\wti a_{1}(\gamma) & 0 & \wti a_{2}(\gamma) & 0 & \ldots & & \\
0 & \wti a_{2}(\gamma) & 0 & \wti a_{3}(\gamma) & 0 & \ldots & \\
\vdots & 0 & \wti a_{3}(\gamma) & 0 & \wti a_{4}(\gamma) & 0 & \ldots \\
& \vdots & 0 & \ddots & \ddots & \ddots & \\
\end{pmatrix}
\begin{pmatrix}
\wti d_{0}(\gamma) \\
\wti d_{1}(\gamma) \\
\wti d_{2}(\gamma) \\
\vdots\\
\phantom{\vdots}\\
\end{pmatrix}   \\ 
& =-\f{1}{\gamma}
\begin{pmatrix}
\wti d_{0}(\gamma) \\
\wti d_{1}(\gamma) \\
\wti d_{2}(\gamma) \\
\vdots\\
\phantom{\vdots}\vspace{1mm}\\
\end{pmatrix}  
= - \f{1}{\gamma} \wti d(\gamma), \quad \gamma \in (0,\gamma_{c,n}],      \lb{4.19} 
\end{split} 
\end{align}
in analogy with \eqref{4.14}, with
\begin{align}
& \,\, \wti a_{\ell}(\gamma) = w_{\ell}(\gamma)^{-1/2} a_{\ell}\,w_{\ell-1}(\gamma)^{-1/2}\no\\
& \,\,\, \qquad= [\ell(\ell+n-2) - \lambda_{\gamma,n,0}]^{-1/2} [(\ell-1)(\ell+n-3) - \lambda_{\gamma,n,0}]^{-1/2} \no\\
& \,\, \, \quad \qquad\times\bigg[\f{\ell(\ell+n-3)}{(2\ell+n-4)(2\ell+n-2)}\bigg]^{1/2}, \quad \ell \in \bbN,	\lb{4.20}  \\
&\wti a_{\ell}(\gamma) \underset{\ell \to \infty}{=} \Oh\big(\ell^{-2}\big).
	\lb{4.21}
\end{align}

Recalling once more that $\lambda_{\gamma_{c,n},n,0} = - (n-2)^2/4$ (cf.\ \eqref{2.31}), and that 
$\lambda_{\gamma,n,0}$ as well as $\Psi_{\gamma}(\dott)$ in \eqref{4.1} are analytic with respect to 
$\gamma$ as $\gamma$ varies in a complex neighborhood of $[0,\infty)$, taking the limit 
$\gamma \uparrow \gamma_{c,n}$ on either side of \eqref{4.19} yields 
\begin{equation}
K(\gamma_{c,n}) \wti d(\gamma_{c,n}) = - \gamma_{c,n}^{-1} \, \wti d(\gamma_{c,n}), 
\end{equation}
implying 
\begin{equation}
- \gamma_{c,n}^{-1} \in \sigma_p(K(\gamma_{c,n})).
\end{equation}

Next, we prove that $- \gamma_{c,n}^{-1}$ is the smallest (negative) eigenvalue of the operator $K(\gamma_{c,n})$ in $\ell^{2}(\bbN_{0})$.

\begin{proposition}	\lb{p4.1}
One has $K(\gamma)\in\cB_{\infty} \big(\ell^{2}(\bbN_{0})\big)$, $\gamma \in (0,\gamma_{c,n}]$, 
and
\begin{equation}
\gamma K(\gamma) \geq - I_{\ell^2(\bbN_0)}, \quad \gamma \in (0,\gamma_{c,n}].      \lb{4.21a} 
\end{equation} 
Moreover, 
\begin{align}
||K(\gamma_{c,n})||_{\cB(\ell^{2}(\bbN_{0}))} \leq 
\begin{cases}
8\big/\big[3^{3/2}\big],\quad& n=3,\\
2\max\{\wti a_{\ell_{r,n}}(\gamma_{c,n}),\wti a_{\ell_{r,n}-1}(\gamma_{c,n})\} ,\quad& n\geq 4,
\end{cases}
	\lb{4.22}
\end{align}
where 
\begin{equation}
r(n) = \f{1}{4}\left(6-2n + \big\{2 \big[26-18n+3n^{2}\big]\big\}^{1/2}\right),  \quad \ell_{r,n}=\lceil r(n)\rceil,    \lb{4.23} 
\end{equation}
with $\lceil x\rceil=\inf\{m\in\bbN_{0}\,|\,m\geq x\}$, the ceiling function.
\end{proposition}
\begin{proof}
The compactness assertion for $K(\gamma)$, $\gamma \in (0,\gamma_{c,n}]$, follows from the limiting behavior 
$\wti a_{\ell}(\gamma) \underset{\ell\to\infty}{\longrightarrow}0$, see, for instance, \cite[p.~201]{Va96}. 

To prove the uniform lower bound \eqref{4.21a} one can argue by contradiction as follows: Fix $\gamma \in (0,\gamma_{c,n}]$ and suppose there exists $\varepsilon > 0$ such that $-(1+\varepsilon) \in 
\sigma(K(\gamma))$, that is, $-(1+\varepsilon) \in \sigma_p(K(\gamma))$. Then working backwards from 
\eqref{4.19} to \eqref{4.1} yields the existence of $\{d_{\ell}(\varepsilon, \gamma)\}_{\ell \in \bbN_0} \in 
\ell^2(\bbN_0; w)$ such that 
\begin{align}
\begin{split} 
& -\f{d^{2}\Psi_{\varepsilon,\gamma} (\theta_{n-1})}{d\theta_{n-1}^{2}} 
- (n-2)\cot(\theta_{n-1})\f{d\Psi_{\varepsilon,\gamma}(\theta_{n-1})}{d\theta_{n-1}} +
\f{\gamma}{1+\varepsilon} \cos(\theta_{n-1}) \Psi_{\varepsilon, \gamma} (\theta_{n-1})    \\
& \quad = \lambda_{\gamma,n,0} \Psi_{\varepsilon, \gamma}(\theta_{n-1}), 
\quad \gamma \in (0,\gamma_{c,n}].      \lb{4.23a}
\end{split} 
\end{align}
where 
\begin{align}
\Psi_{\varepsilon, \gamma} (\theta_{n-1}) &=\sum\limits_{\ell=0}^{\infty} d_{\ell}(\varepsilon, \gamma) \bigg[\f{\ell !(2\ell+n-2)}{2^{4-n}\pi\Gamma(\ell+n-2)}\bigg]^{1/2}\Gamma((n-2)/2) C_{\ell}^{(n-2)/2}(\cos(\theta_{n-1})).     \lb{4.23b}
\end{align}
By the strict monotonicity of $ \lambda_{\gamma,n,0}$ with respect to $\gamma \in (0,\gamma_{c,n}]$ one infers that 
\begin{equation}
\lambda_{\gamma,n,0} < \lambda_{\gamma/(1 + \varepsilon),n,0},     \lb{4.23c} 
\end{equation}
and hence \eqref{4.23a} contradicts the fact that by definition, $\lambda_{\gamma/(1 + \varepsilon),n,0}$ is the lowest eigenvalue of $\Lambda_{\gamma/(1+\varepsilon),\cL^{n}}$.

To obtain the bound \eqref{4.22}, \eqref{4.23}, one applies \cite[Theorem 1.5]{Te99} after calculating 
$||\wti a(\gamma_{c,n})||_{\ell^{\infty}(\bbN_{0})}$.  
As $\wti a_{\ell,n}(\gamma_{c,n})$ is bounded and tends to 0 as $\ell\to\infty$ for all $n\geq 3$, it attains its supremum.  To find the index where this occurs, one considers $\ell$ as a continuous variable, and solves 
$\f{d \wti a_{\ell,n}(\gamma_{c,n})}{d\ell}=0$.  The value $r(n)$ emerges as the only nonnegative, real root of this expression, but as $0<r(n)\in\bbR\backslash\bbN$ for $n\geq 4$, the maximum in \eqref{4.22} and the ceiling function in \eqref{4.23} are required.  Since $r(3)\in\bbC\backslash\bbR$, the norm 
$||\wti a_{\,\cdot\,,3}(\gamma_{c,n})||_{\ell^{\infty}(\bbN_{0})}$ must be computed separately as 
$||\wti a_{\,\cdot\,,3}(\gamma_{c,n})||_{\ell^{\infty}(\bbN_{0})}=|\wti a_{1,3}(\gamma_{c,n})| = 8\big/\big[3^{3/2}\big]$.
\end{proof}

To introduce the notion of finite truncations, one considers the operators
\begin{align}
P_{m}=
\begin{pmatrix}
I_{\bbC^m} & 0\\
0 & 0
\end{pmatrix}
=I_{\bbC^m} \oplus 0,\quad K_{m}(\gamma_{c,n}) = P_{m}K(\gamma_{c,n})P_{m}, \quad m \in \bbN,
	\lb{4.24}
\end{align}
on $\ell^{2}(\bbN_{0})$ (with $I_{\bbC^m}$ denoting the identity matrix in $\bbC^{m}$, $m \in \bbN$).  

We also introduce the finite $N \times N$ tri-diagonal Jacobi matrices $J_N(\wti a_0,\dots,\wti a_{N-1})$ in $\bbC^N$, $N \in \bbN$, $N \geq 2$, 
denoted by 
\begin{equation}
J_N(\wti a_1,\dots, \wti a_{N-1}) = 
\begin{pmatrix}
0 & \wti a_{1} & \phantom{0} & \phantom{0} & \phantom{0} & \phantom{0} \\
\wti a_{1} & 0 & \wti a_{2} & \phantom{0} & \bf{0} & \phantom{0} \\
\phantom{0} & \wti a_{2} & 0 & \wti a_{3} & \phantom{0} & \phantom{0} \\
\phantom{0} & \phantom{0} & \wti a_{3} & 0 & \ddots & \phantom{0} \\
\phantom{0} & \bf{0} & \phantom{0} & \ddots & \ddots & \wti a_{N-1} \\
\phantom{0} & \phantom{0} & \phantom{0} & \phantom{0} & \wti a_{N-1} & 0 \\
\end{pmatrix}, 
\quad N \in \bbN, N \geq 2,
	\lb{4.25}   
\end{equation} 
in particular, 
\begin{align}
& {\det}_{\bbC^N}(z I_N - J_N(\wti a_1,\dots, \wti a_{N-1})) 
= z {\det}_{\bbC^{N-1}} (z I_{N-1} - J_{N-1}(\wti a_2,\dots, \wti a_{N-1})) 
\no \\
& \qquad - [\wti a_1]^2 {\det}_{\bbC^{N-2}} (z I_{N-2} - J_{N-2}(\wti a_3,\dots, \wti a_{N-1})), 
\quad z \in \bbC,  \no\\
& \quad = \begin{cases} 
z P_{(N-1)/2}\big(z^2\big), \quad P_{(N-1)/2} (0) \neq 0, & N \text{ odd,} \\
Q_{N/2}\big(z^2\big), \quad Q_{N/2}(0) \neq 0, & N \text{ even,} 
\end{cases}  
	\lb{4.26}
\end{align} 
where $P_{(N-1)/2}(\,\cdot\,)$ and $Q_{N/2}(\,\cdot\,)$ are monic polynomials of degree $(N-1)/2$ and $N/2$, respectively. 

Thus, the spectrum of each $J_N(\wti a_0,\dots, \wti a_{N-1})$ consists of $N$ real eigenvalues, symmetric with respect to the origin, the eigenvalues being simple as long as $\wti a_j > 0$, $1 \leq j \leq N-1$ (see, e.g., \cite[Theorem~II.1.1]{GK02}, \cite[Remark~1.10 and p.~120]{Te99}). 
Explicitly,
\begin{align}
& {\det}_{\bbC^2}((z I_{\bbC^2} - J_2(\wti a_1)) = z^2 - [\wti a_1]^2,   \no \\
& {\det}_{\bbC^3}((z I_{\bbC^3} - J_3(\wti a_1,\wti a_2)) 
= z \big\{z^2 - [\wti a_1]^2 - [\wti a_2]^2\big\},   \no  \\
& {\det}_{\bbC^4}((z I_{\bbC^4} - J_4(\wti a_1, \wti a_2, \wti a_3)) 
= z^4 - \big\{[\wti a_1]^2 + [\wti a_2]^2 + [\wti a_3]^2\big\} z^2 
+ [\wti a_1]^2 [\wti a_3]^2,     \lb{4.27} \\
& {\det}_{\bbC^5}((z I_{\bbC^5} - J_5(\wti a_1, \wti a_2, \wti a_3, \wti a_4)) 
= z \big\{z^4 - \big\{[\wti a_1]^2 + [\wti a_2]^2 + [\wti a_3]^2 + [\wti a_4]^2\big\} z^2 \no \\
& \hspace*{5.55cm} + [\wti a_1]^2 [\wti a_3]^2 + [\wti a_1]^2 [\wti a_4]^2 
+ [\wti a_2]^2 [\wti a_4]^2\big\},    \no \\
& \quad \text{etc.}   \no 
\end{align}

In addition, we introduce the unitary, self-adjoint, diagonal operator $W$ in $\ell^2(\bbN_0)$ as 
\begin{equation}
W = \big((-1)^p \delta_{p,q}\big)_{(p,q) \in \bbN_0^2}, \quad W^{-1} = W = W^*.
	\lb{4.28}
\end{equation}

\begin{theorem}	\lb{t4.2}
Given the operators $K(\gamma_{c,n})$, $K_{m}(\gamma_{c,n})$, $m \in \bbN$, and $W$ as in \eqref{4.19}, 
\eqref{4.24}--\eqref{4.28}, one concludes that 
$K(\gamma_{c,n})$ and $-K(\gamma_{c,n})$ as well as $K_m(\gamma_{c,n})$ and $- K_m(\gamma_{c,n})$ are unitarily equivalent,  
\begin{equation}
- K(\gamma_{c,n}) = W K(\gamma_{c,n}) W^{-1}, \quad 
- K_m(\gamma_{c,n}) = W K_m(\gamma_{c,n}) W^{-1}, \; m \in \bbN,    \lb{4.29} 
\end{equation}
and hence the spectra of $K(\gamma_{c,n})$ and $K_m(\gamma_{c,n})$, $m \in \bbN$, are symmetric with respect to zero. Moreover, all nonzero eigenvalues of $K(\gamma_{c,n})$ and $K_m(\gamma_{c,n})$, 
$m \in \bbN$, are simple. In addition,  
\begin{equation} 
\lim_{m \to \infty} \|K_{m}(\gamma_{c,n}) - K(\gamma_{c,n})\|_{\cB(\ell^2(\bbN_0))} =0,   \lb{4.30}
\end{equation} 
and\footnote{Here $\sigma_{ess}(\, \cdot \,)$ denotes the essential spectrum.}
\begin{align}
\begin{split} 
& \sigma(K(\gamma_{c,n}))=\underset{m\to\infty}\lim\sigma(K_{m}(\gamma_{c,n})),  \\
& \sigma_{ess}(K(\gamma_{c,n})) = \sigma_{ess}(K_{m}(\gamma_{c,n})) = \{0\}, \; m \in \bbN.     \lb{4.31}
\end{split} 
\end{align}
In particular, $\lambda\in\sigma(K(\gamma_{c,n}))$ if and only if there is a sequence 
$(\lambda_{m})_{m\in\bbN}$ with $\lambda_{m}\in\sigma(K_{m}(\gamma_{c,n}))$ such that 
$\lambda_{m}\underset{m\to\infty}{\longrightarrow}\lambda$.
\end{theorem}
\begin{proof}
The symmetry fact \eqref{4.29} follows from an elementary computation. That all eigenvalues of $K(\gamma_{c,n})$ are simple follows from the fact that $K(\gamma_{c,n})$ is a half-lattice operator with $\wti a_{\ell}(\gamma_{c,n}) > 0$, $\ell \in \bbN$, and hence the half-lattice does not decouple into a disjoint union of subsets (resp., $K(\gamma_{c,n})$ does not reduce to a direct sum of operators in $\ell^2(\bbN_0)$). The same argument applies to the finite-lattice operators $K_m(\gamma_{c,n})$, $m \in \bbN$.

One notices that $\slim_{m \to \infty} P_{m} = I_{\ell^2(\bbN_0)}$, where strong operator convergence is abbreviated by $\slim$. Together with the compactness of $K(\gamma_{c,n})$ given in Proposition \ref{p4.1}, one obtains 
\begin{equation}
\lim_{m\to\infty} \|P_{m}K(\gamma_{c,n}) - K(\gamma_{c,n})\|_{\cB(\ell^2(\bbN_0))} = 0,
   \lb{4.32} 
\end{equation}
applying \cite[Proposition 3.11]{Am80}.~The norm convergence in \eqref{4.32}, together with the uniform bound 
$\|P_m\|_{\cB(\ell^2(\bbN_0))} = 1$, 
$m \in \bbN$, yields \eqref{4.30}. The latter implies \eqref{4.31} as a consequence of \cite[Theorem~VIII.23\,(a) and 
Theorem~VIII.24\,(a)]{RS80} (see also \cite[Satz~9.24\,$a)$]{We00}), taking into account that norm resolvent convergence of a sequence of self-adjoint operators is equivalent to norm convergence of a uniformly bounded sequence of self-adjoint operators in a complex Hilbert space (see 
\cite[Theorem~VIII.18]{RS80}, \cite[Satz~9.22\,a)\,$(ii)$]{We00}). 
\end{proof}

Returning to the dipole context, one may now compute approximants of $\gamma_{c,n}$ by approximating the smallest negative eigenvalues of $K(\gamma_{c,n})$ in terms of the smallest negative eigenvalue of $K_{m}(\gamma_{c,n})$ with increasing $m \in \bbN$. Using $K_{7}$, (which produced 16 stable digits in the case $n=3$), one obtains the following values and approximants for $3 \leq n \leq 10$:

\medskip 

\begin{center}
\Large
 \begin{tabular}{||c | c | c | c | c||}
 \hline
 $n$ & lower bound for $\gamma_{c,n}$ & $\gamma_{c,n}$ & upper bound $\gamma_{c,n}$ \\ [0.5ex] 
 \hline 
 3 & 0.250 & 1.279 & 4.418 \\ 
 \hline
 4 & 1.000 & 3.790 & 5.890 \\
 \hline
 5 & 2.598 & 7.584 & 10.308 \\
 \hline
 6 & 5.846 & 12.672 & 17.672 \\
 \hline
 7 & 10.392 & 19.058 & 27.980 \\ 
 \hline
 8 & 16.238 & 26.742 & 41.233 \\
 \hline
 9 & 23.383 & 35.725 & 57.432 \\
 \hline
 10 & 31.826 & 46.006 & 76.576 \\ 
 \hline
\end{tabular}
\end{center}

\begin{figure}[h]
    \centering
    \includegraphics[width=0.96\textwidth]{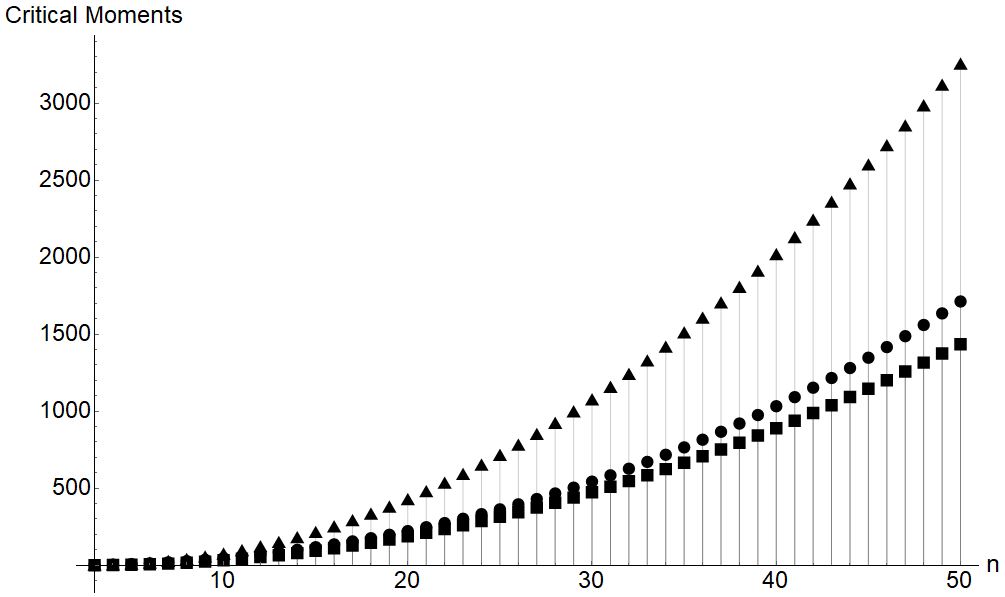}
    \caption{Dimension $n$ vs. Critical Dipole Moment $\gamma_{c,n}$ (\textbullet) with its upper ($\blacktriangle$) and lower ($\blacksquare$) bounds.}
    \label{fig:ngcb}
\end{figure}

Here the lower and upper bounds for $\gamma_{c,n}$ correspond to the values displayed in \eqref{3.41e} (see Fig.~1). 

The result for $\gamma_{c,3}$ is in excellent agreement with the ones found in the literature (see, e.g., 
\cite{AEG78}, \cite{BR67}, \cite{CG07}, \cite{Cr67}, \cite{FT47}, \cite{Le67}, \cite{Tu77}, and \cite{TF66}). 
The approximate values of $\gamma_{c,n}$ for $n \geq 4$ (but surprisingly, not for $n=3$) are in good agreement with those obtained in \cite[p.~98]{FMT08}.

\section{Multicenter Extensions} \lb{s5}

Combining the results of this manuscript with those in \cite{GMNT16} one can extend the scope of this investigation to include multicenter dipole interactions, that is, sums of point dipoles supported on an infinite discrete set (a set of distinct points spaced apart by a minimal distance $\varepsilon > 0$). For various related studies on multicenter singular interactions, see, for instance, \cite{BDE08}, \cite{Ca16}, \cite{Du06}, \cite{FMT07}, \cite{FMT08}, \cite{FMT09}, 
\cite{FMO20}, \cite{GMNT16}, \cite{HG80}, \cite{LP20}, \cite{Te96}. 

To set the stage, we recall the notion of relative form boundedness in the special context of self-adjoint  operators.

\begin{definition} \lb{d5.1}
Suppose that $A$ is self-adjoint in a complex Hilbert space $\cH$ and bounded from below, that is,
$A \ge c I_{\cH}$ for some $c\in\R$. Then the sesquilinear form $Q_A$ associated with $A$ is denoted by
\begin{equation}
Q_A(f,g) = \big((A - c I_{\cH})^{1/2} f, (A - c I_{\cH})^{1/2} g\big)_{\cH} + c (f,g)_{\cH}, \quad 
f \in \dom\big(|A|^{1/2}\big).     \lb{5.1} 
\end{equation}
A sesquilinear form $q$ in $\cH$ satisfying $\dom(q) \supseteq \dom(Q_A)$ is called {\it bounded with respect to the form $Q_A$} if for some $a, b \geq 0$,
\begin{equation}
|q(f,f)| \leq a \, Q_A(f,f) + b \, \|f\|_{\cH}^2, \quad f \in \dom(Q_A). 
\end{equation}
The infimum of all numbers $a$ for which there exists $b \in [0,\infty)$ such that \eqref{5.1} holds is called the bound of $q$ with respect to $Q_A$. 
\end{definition}

\medskip

The following result is a variant of \cite[Theorem~3.2]{GMNT16}, which in turn is an abstract version of 
Morgan \cite[Theorem~2.1]{Mo79} (see also \cite[Proposition~3.3]{CFKS87}, \cite{Is61}, \cite[Sect.\ 4]{Ki81}). Throughout this section, infinite sums are understood in the weak operator topology and $J \subseteq \N$ 
denotes an index set.

\begin{lemma} \lb{l5.2}  
Suppose that $T$ is a self-adjoint operator in $\cH$ bounded from below, $T \geq c I_{\cH}$ for some 
$c \in \bbR$, and $W$ is a self-adjoint operator in $\cH$ such that  
\begin{equation} 
\dom\big(|T|^{1/2}\big)\subseteq\dom\big(|W|^{1/2}\big).     \lb{5.3} 
\end{equation} 
We abbreviate 
\begin{equation}
q_W(f,g) = \big(|W|^{1/2}f, \sgn(W) |W|^{1/2}g\big)_{\cH}, \quad f \in \dom\big(|W|^{1/2}\big).
\end{equation} 
Let $d, D \in (0,\infty)$, $e \in [0,\infty)$, assume that $\Phi_j \in \cB(\cH)$, $j \in J$, leave 
$\dom\big(|T|^{1/2}\big)$ invariant, that is, 
\begin{equation} 
\Phi_j \dom\big(|T|^{1/2}\big) \subseteq \dom\big(|T|^{1/2}\big), \quad j \in J, 
\end{equation} 
and suppose that the following conditions $(i)$--$(iii)$ hold: \\[.5mm]  
$(i)$ $\sum_{j  \in J} \Phi_j^* \Phi_j \leq I_{\cH}$. \\[1mm] 
$(ii)$ $\sum_{j \in J} |q_W(\Phi_jf, \Phi_jf)| \geq D^{-1} |q_W(f,f)|$, \, $f \in \dom\big(|T|^{1/2}\big)$. \\[1mm] 
$(iii)$ $\sum_{j  \in J}  \| |T|^{1/2} \Phi_j f\|_{\cH}^2 \le d\| |T|^{1/2} f\|_{\cH}^2 
+ e\|f\|_{\cH}^2$, \, $f \in \dom\big(|T|^{1/2}\big)$. \\[1mm] 
Then,
\begin{equation}
|q_W(\Phi_j f,\Phi_j f)| \leq a \, q_T(\Phi_j f, \Phi_j f)
+ b \|\Phi_j f\|_{\cH}^2, \quad f \in\dom(|T|^{1/2}), \; j\in J,   \lb{5.6}
\end{equation} 
implies
\begin{equation}
|q_W(f,f)|  \leq a\,d\,D \,q_T(f,f) + [a\,e + b]D \,\|f\|_{\cH}^2, \quad f \in\dom(|T|^{1/2}). 
\end{equation} 
\end{lemma}
\begin{proof}
For $f \in \dom\big(|T|^{1/2}\big)$ one computes
\begin{align} 
|q_W(f,f)|  &\leq D \sum_{j  \in J} |q_W(\Phi_j f,\Phi_j f)|  
\quad \text{(by $(ii)$)}   \no \\ 
&\leq D \sum_{j  \in J}  \left[ a \big\||T|^{1/2} \Phi_j f \big\|_{\cH}^2 + b \|\Phi_j f\|_{\cH}^2 \right]
\quad \text{(by \eqref{5.6})}   \no \\ 
&\leq a\, d \, D \big\|  |T|^{1/2} f \big\|_{\cH}^2 + (a\,e\,D + b\,D) \|f\|_{\cH}^2 
\quad \text{(by $(i)$ and $(iii)$)}     \no \\ 
& = a\,d\,D \,q_T(f,f) + [a\,e + b]D \,\|f\|_{\cH}^2,     
\end{align}
completing the proof.
\end{proof}

\begin{remark} \lb{r5.3}
Considering the concrete case of 
\begin{align}
& H_0 = - \Delta, \quad \dom(H_0) = H^2(\R^n),    \no \\
& \, Q_{H_0}(f,g) = \big((-\Delta)^{1/2} f, (-\Delta)^{1/2} g\big)_{L^2(\bbR^n)} 
= (\nabla f, \nabla g)_{[L^2(\bbR^n)]^n},     \lb{5.9} \\ 
& \hspace*{5.2cm} f, g \in \dom(Q_{H_0}) = H^1(\bbR^n),    \no 
\end{align}
in $L^2(\R^n)$, and assuming that $W$, the operator of multiplication with a measurable and 
a.e.~real-valued function 
$W(\, \cdot \,)$, employing a slight abuse of notation, satisfies \eqref{5.3} (for sufficient conditions on $W$, see, e.g., \cite[Theorems~10.17\,(b), 10.18]{We80} with $r=1$).
Let $\{\phi_j\}_{j\in J}$, $J\subseteq \N$, be a family of smooth, real-valued functions defined 
on $\R^n$ in such a manner that for each $x \in \R^n$, there exists an open neighborhood $U_x \subset \R^n$ of $x$ such that there exist only finitely many indices $k \in J$ with 
$\supp \, (\phi_k) \cap U_x \neq \emptyset$ and  $\phi_k|_{U_x} \neq 0$, as well as 
\begin{equation}
\sum_{j\in J} \phi_j(x)^2=1, \quad x \in \R^n    \lb{5.10}
\end{equation}
(the sum over $j \in J$ in \eqref{5.10} being finite). Finally, let $\Phi_j$ be the operator of 
multiplication by the function $\phi_j$, 
$j\in J$. Then one notes that for these choices, hypothesis $(i)$ holds with equality, and 
hypothesis $(ii)$ with $D=1$ follows from $(i)$. 
Moreover, item $(iii)$ holds with $d=1$ as long as 
\begin{equation}
e = \bigg\| \sum_{j \in J}  |\nabla \phi_j(\cdot)|^2 \bigg\|_{L^{\infty}(\R^n)} < \infty.   \lb{5.11} 
\end{equation}
To verify this, one observes that 
$\| |H_0|^{1/2} \phi f\|_{L^2(\R^n)}^2 = \int_{\R^n} d^n x \, |\nabla (\phi(x) f(x))|^2$ 
and that the cross terms vanish since $\sum_{j \in J}  \phi_j(x) (\nabla \phi_j)(x) =0$, 
$x \in \R^n$, by condition \eqref{5.10}. (We note again that the latter sum over $j \in J$ contains only 
finitely many terms in every bounded neighborhood of $x \in \R^n$.) \hfill $\diamond$
\end{remark}

Strongly singular potentials that are covered by Lemma \ref{l5.2} are, for instance, of the 
following form: Let $J\subseteq \N$ be an index set, $\varepsilon > 0$, and $\{x_j\}_{j\in J}\subset\R^n$, 
$n \in \N$, $n \geq 3$, be a set of points such that 
\begin{equation}
\inf_{\substack{j, j' \in J \\ j \neq j'}} |x_j - x_{j'}| \geq \varepsilon.    \lb{5.12} 
\end{equation} 
In addition, let $\gamma_j \in \R$, $j \in J$, $\gamma_0 \in [0, \infty)$, with 
\begin{equation}
|\gamma_j| \leq \gamma_0 < (n -2)^2/4, \quad j \in J,    \lb{5.13} 
\end{equation}
and 
\begin{equation}
W_{\{\gamma_j\}_{j \in J}}(x) = \sum_{j \in J} \gamma_j |x - x_j|^{-2} \chi_{B_n(x_j; \varepsilon/4)}(x) 
+ W_0(x),     \lb{5.14} 
\quad x \in \R^n \backslash \{x_j\}_{j \in J},    
\end{equation}
with 
\begin{equation}
W_0 \in L^{\infty}(\bbR^n), \, \text{ $W_0$ real-valued~a.e.~on $\bbR^n$,}    \lb{5.14a} 
\end{equation}
and $B_n(x_0;r)$ the open ball in $\bbR^n$ of radius $r>0$, centered at $x_0 \in \bbR^n$.

Then an application of Hardy's inequality in $\R^n$, $n \geq 3$ (cf.\ \eqref{1.1}), shows that 
$W_{\{\gamma_j\}_{j \in J}}$ is form bounded with respect to $T_0$ in \eqref{5.9} with form bound 
strictly less than one.

At this point one can extend existing results of \cite{FMT08}, \cite{FMT09} regarding quadratic form estimates for multicenter dipole interactions as follows. 

\begin{theorem} \lb{t5.4}
Given \eqref{5.9}--\eqref{5.12} and $W_0$ in \eqref{5.14a}, we introduce $\gamma_0, \gamma_j \in [0,\infty)$, $j \in J$, satisfying  
\begin{equation}
0 \leq \gamma_j \leq \gamma_0 < \gamma_{c,n},  \quad j \in J,    \lb{5.15} 
\end{equation}
and 
\begin{align}
\begin{split} 
q_{\{\gamma_j\}_{j \in J}}(f,g) &= \sum_{j \in J} \gamma_j \int_{\bbR^n} d^nx \, 
(u, (x-x_j)) |x-x_j|^{-3} \chi_{B_n(x_j;\varepsilon/4)}(x) \ol{f(x)} g(x),    \\ 
& \quad + \int_{\bbR^n} d^n x \, W_0 (x) \ol{f(x)} g(x),  \quad f, g \in H^1(\bbR^n).  
\end{split} 
\end{align}
Then $q_{\{\gamma_j\}_{j \in J}}$ is bounded with respect to $Q_{T_0}$ in \eqref{5.9} with form bound 
strictly less than one. 
\end{theorem}
\begin{proof}
Without loss of generality we put $W_0 = 0$.
Inequality \eqref{3.2} and the analogous inequality with $u \in \bbR^n$, $|u| =1$, replaced by $-u$ yields 
\begin{align} 
\begin{split} 
&\text{for all $\gamma \in [0,\gamma_{c,n}]$,} \\
& \quad \int_{\bbR^n} d^n x \, |(\nabla f)(x)|^2 \geq \pm \gamma \int_{\bbR^n} d^n x \, (u, x) |x|^{-3} |f(x)|^2, 
\quad f \in H^1(\bbR^n),     \lb{5.17}
\end{split} 
\end{align}  
and hence for some $\varepsilon_{\gamma_0} \in (0,1)$ and $c(\varepsilon_{\gamma_0}) \in (0,\infty)$, 
\begin{align} 
\begin{split} 
& \pm \gamma_0 \int_{\bbR^n} d^n x \, (u, x) |x|^{-3} |f(x)|^2 \leq 
(1 - \varepsilon_{\gamma_0}) \big\|\nabla f\big\|_{[L^2(\bbR^n)]^n}^2    \\
& \quad + c(\varepsilon_{\gamma_0}) \|f\|_{L^2(\bbR^n)}^2,  \quad  
\gamma_0 \in[0, \gamma_{c,n}), \; f \in H^1(\bbR^n).     \lb{5.18}
\end{split} 
\end{align}

Thus, an application of Lemma \ref{l5.2}, taking into account that $d=D=1$ as described in 
Remark \ref{r5.3}, implies for some $C(\varepsilon_{\gamma_0}) \in (0,\infty)$, 
\begin{align}
\begin{split} 
|q_{\{\gamma_j\}_{j \in J}}(f,f)| \leq (1 - \varepsilon_{\gamma_0}) \big\|\nabla f\big\|_{[L^2(\bbR^n)]^n}^2 
+ C(\varepsilon_{\gamma_0}) \|f\|_{L^2(\bbR^n)}^2,&  \\
f \in H^1(\bbR^n),&  
\end{split} 
\end{align}
as was to be proven. 
\end{proof}

Thus, Theorem \ref{t5.4} proves semiboundedness of the self-adjoint multicenter dipole Hamiltonian 
$H_{\{\gamma_j\}_{j \in J}}$ in $L^2(\bbR^n)$, uniquely associated with the quadratic form sum
\begin{equation}
Q_{H_{\{\gamma_j\}_{j \in J}}} = Q_{T_0} + q_{\{\gamma_j\}_{j \in J}}, 
\quad \dom(Q_{H_{\{\gamma_j\}_{j \in J}}}) = H^1(\bbR^n), 
\end{equation}
under very general hypotheses on  $\{x_j\}_{j\in J}\subset\R^n$ and $\{\gamma_j\}_{j \in J}$. We note that \cite{FMT08}, \cite{FMT09} derive sufficient conditions $\{x_j\}_{j\in J}\subset\R^n$ and $\{\gamma_j\}_{j \in J}$ to guarantee nonnegativity of $H_{\{\gamma_j\}_{j \in J}}$ and also discuss situations characterized by the lack of nonnegativity of $H_{\{\gamma_j\}_{j \in J}}$. 

Finally, we sketch how Remark \ref{r3.6} extends to the multicenter situation.

\begin{remark} \lb{r5.5}
Assume \eqref{5.9}--\eqref{5.12}, let $W_0$ as in \eqref{5.14a}, suppose $\gamma_j \in [0,\infty)$, 
and introduce 
\begin{align}
& V(x) = \sum_{j \in J} \big[\gamma_j (u,(x-x_j)) |x-x_j|^{-3} + \wti V_j(|x-x_j|)\big] 
\chi_{B_n(x_j; \varepsilon/4)}(x) + W_0(x),     \no \\ 
& \hspace*{8.2cm} x \in \R^n \backslash \{x_j\}_{j \in J},    
\end{align}
with 
\begin{equation}
r \wti V_j(r) \in L^1((0,\varepsilon); dr) \cap L^{\infty}_{loc}((0,\varepsilon]; dr), \quad j \in J.  
\end{equation}
Consider the minimally defined Schr\"odinger operator $\dot H_{\{\gamma_j\}_{j \in J}}$ in 
$L^2(\bbR^n)$ given by
\begin{equation}
\dot H_{\{\gamma_j\}_{j \in J}} = - \Delta + V(\, \cdot \,), \quad 
\dom\big(\dot H_{\{\gamma_j\}_{j \in J}}\big) = C_0^{\infty}(\bbR^n \backslash \{x_j\}_{j \in J}). 
\end{equation}
Then the criterion \eqref{3.51} and \eqref{2.34}, \eqref{2.35} combined with 
\cite[Theorems~1.1 and 5.8]{GMNT16} imply that 
\begin{align}
\begin{split}
& \text{$\dot H_{\{\gamma_j\}_{j \in J}}$ is essentially self-adjoint} \\
& \quad \text{if and only if $\lambda_{\gamma_j,n,0} 
\geq - n(n-4)/4$ for each $j \in J$.} 
\end{split} 
\end{align}
${}$ \hfill $\diamond$ 
\end{remark}

\appendix

\section{Spherical Harmonics and the Laplace--Beltrami Operator in $L^2\big(\bbS^{n-1}\big)$, $n \geq 2$.} \lb{sA} 

In this appendix we summarize some of the results on spherical harmonics and the Laplace--Beltrami operator on the unit sphere $\bbS^{n-1}$ in dimensions $n \in \bbN$, $n \geq 2$, following 
\cite[Chs.~2,3]{AH12}, \cite[Ch.~1]{DX13}, and \cite[Ch.~2]{He99}.

Assuming $n \in \bbN$, $n \geq 2$, cartesian and polar coordinates (cf.\, e.g., \cite{Bl60}) on $\bbS^{n-1}$ are given by 
\begin{align} 
& x = (x_1,\dots,x_n) \in \bbR^n,    \no \\
&  x = r \omega, \; \omega = \omega(\theta) = \omega(\theta_{1},\theta_{2},\dots,\theta_{n-1}) 
= x/|x| \in \bbS^{n-1},   \lb{A.1} \\
& x_k \in \bbR, \, 1 \leq k \leq n, \; r = |x| \in[0,\infty), \; \theta_{1}\in[0,2\pi), \; \theta_{j}\in[0,\pi), \, 2 \leq j\leq n-1,   \no 
\end{align} 
where (cf., e.g., \cite{Bl60}, \cite[Sect.~1.5]{DX13}) 
\begin{equation}
\begin{cases}
x_{1}=r \cos(\theta_1) \prod\limits_{j=2}^{n-1} \sin(\theta_j),   \\[1mm] 
x_{2}=r \sin(\theta_1) \prod\limits_{j=2}^{n-1} \sin(\theta_j),   \\
\; \vdots\\
x_{n-1}=r \cos(\theta_{n-2}) \sin(\theta_{n-1}), \\
x_{n}=r \cos(\theta_{n-1}).
\end{cases}
	\lb{A.2}
\end{equation}
The surface measure $d^{n-1}\omega$ on $\bbS^{n-1}$ and the volume element in $\bbR^n$ then read
\begin{equation}
d^{n-1}\omega(\theta) = d \theta_1 \prod_{j=2}^{n-1}[\sin(\theta_j)]^{j-1} d\theta_j, 
\quad d^{n}x=r^{n-1}dr\,d^{n-1}\omega (\theta),    \lb{A.3}
\end{equation}
in particular, the area $\omega_n$ of the unit sphere $\bbS^{n-1}$ in $\bbR^n$ is given by (cf.\  \cite[p.~2]{Mu66})
\begin{align}
\omega_{n}=\int_{\bbS^{n-1}}d^{n-1}\omega (\theta) = 2\pi^{n/2}/\Gamma(n/2).
	\lb{A.4}
\end{align}

Turning to spherical harmonics next, we recall that a  homogeneous polynomial $P(x_{1},\ldots,x_{n})$ of degree $\ell \in \bbN_0$ (in $n$ variables) satisfies $P(tx_{1},\ldots,tx_{n})=t^{n}P(x_{1},\ldots,x_{n})$ and is a linear combination of terms of degree $\ell$.  The space of such polynomials with real coefficients  is denoted $\mathscr{P}_{\ell}^{n}$.  We define the harmonic homogeneous polynomials of degree $\ell$ in $n$ variables by
\begin{align}
\mathscr{H}_{\ell}^{n} = \big\{P\in\mathscr{P}_{\ell}^{n} \, \big| \, \Delta P=0\big\},
	\lb{A.5}
\end{align}
where $ \Delta$ represents the Laplace differential expression on $\bbR^n$.  Restricting the elements of $\mathscr{H}_{\ell}^{n}$ to the sphere $\bbS^{n-1}$, one obtains $\cY_{\ell}^{n}$, the space of spherical harmonics of degree $\ell$ in $n$ dimensions.  Spaces of different degrees are orthogonal with respect to the real inner product on the sphere, 
\begin{align} 
\begin{split} 
(Y,Z)_{L^{2}(\bbS^{n-1})}=\int_{\bbS^{n-1}} d^{n-1}\omega (\theta) \, Y(\theta) Z(\theta)  =0,&    \\
Y\in\cY_{\ell}^{n}, \, Z\in\cY_{\ell'}^{n}, \; \ell, \ell' \in \bbN_0, \, \ell\neq\ell'.&  \lb{A.6}	
\end{split} 
\end{align}
The dimension of $\cY_{\ell}^{n}$ equals that of $\mathscr{H}_{\ell}^{n}$ and is given by
(\cite[Corollary~1.1.4]{DX13}) 
\begin{align}
\dim(\mathscr{H}_{\ell}^{n})=\binom{\ell+n-1}{\ell}-\binom{\ell+n-3}{\ell-2}=\f{2\ell+n-2}{\ell+n-2}\binom{\ell+n-2}{n-2},      \lb{A.7}
\end{align}
where we use the convention that the second binomial coefficient equals 0 when 
$\ell=0,1$, and replace the final fraction by 1 in the case where $n=2$ and $\ell=0$.  This is equivalently formulated in \cite[Lemma~3, p.~4]{Mu66} as the generating series
\begin{align}
\f{1+x}{(1-x)^{n-1}}=\sum\limits_{\ell=0}^{\infty} \dim(\cY_{\ell}^{n}) \, x^{\ell}. 
	\lb{A.8}
\end{align}

Most importantly, the spherical harmonics are the eigenfunctions of the 
Laplace--Beltrami operator $ \Delta_{\bbS^{n-1}}$ in $L^2\big(\bbS^{n-1}\big)$, satisfying the eigenvalue equation
\begin{align}
(- \Delta_{\bbS^{n-1}}Y)(\theta)=\ell(\ell+n-2)Y(\theta),\quad Y\in\cY_{\ell}^{n}, \; \ell \in \bbN_0.
	\lb{A.9}
\end{align}

Following \cite[Sect.~1.5]{DX13} an explicit characterization for the spherical harmonics reads as follows: Introducing the multi-index $\alpha=(\alpha_{1},\ldots,\alpha_{n})\in\bbN_{0}^{n}$, with 
$|\alpha|=\sum\limits_{j=1}^{n}\alpha_{j}$, and $\theta = (\theta_{1},\ldots,\theta_{n-1})$, the spherical harmonics are of the form
\begin{align}
Y_{\alpha}(\theta)=[N_{\alpha}]^{-1} g_{\alpha}(\theta_{1})\prod\limits_{j=1}^{n-2}[\sin(\theta_{n-j})]^{|\alpha^{j+1}|}C_{\alpha_{j}}^{\nu_{j}}(\cos(\theta_{n-j})),
	\lb{A.10}
\end{align}
where 
\begin{align} 
&|\alpha^{j}|=\sum\limits_{k=j}^{n-1}\alpha_{k}, \quad \,\nu_{j}=|\alpha^{j+1}|+ [(n-j-1)/2],	\lb{A.11} \\
&\,g_{\alpha}(\theta_{1})=\begin{cases} \cos(\alpha_{n-1}\theta_{1}), & \alpha_{n}=0, \lb{A.12} \\
\sin(\alpha_{n-1}\theta_{1}), &\alpha_{n}=1,\end{cases}    \\
& [N_{\alpha}]^{2}=b_{\alpha}\prod_{j=1}^{n-2}\f{[\alpha_{j}!] ([(n-j+1)/2])_{|\alpha^{j+1}|}(\alpha_{j}+\nu_{j})}{(2\nu_{j})_{\alpha_{j}}([(n-j)/2])_{|\alpha^{j+1}|}\nu_{j}}, \quad
b_{\alpha}=\begin{cases} 2, &\alpha_{n-1}+\alpha_{n}>0,\\ 1, & \text{otherwise.} \end{cases}
	\lb{A.13}
\end{align}
Here the Pochhammer symbol $(x)_{a}$ is defined by
\begin{equation}
(x)_0 =1, \quad (x)_{n} = \Gamma(x +n)/\Gamma(n) = x(x+1) \cdots (x+n-1), \quad n \in \bbN, 
\lb{A.14} 
\end{equation}
and $C^{\lambda}_n(\, \cdot \,)$ represent the Gegenbauer (or ultrasperical) polynomials, see, for instance,  \cite[Ch.~22]{AS72}, \cite[Appendix~B]{DX13}. 

The set $\{Y_{\alpha} \, | \,  |\alpha|=\ell,\alpha_{n}=0,1\}$ represents an orthonormal basis of $\cY_{\ell}^{n}$.

Finally, we recall the expression of the Laplace--Beltrami differential expression on $\bbS^{n-1}$ in spherical coordinates.  From \cite[p.~94]{AH12}, \cite[Lemma~1.4.2]{DX13}, one obtains the recursion\footnote{For clarity we indicate the space dimension $n \in \bbN$ as a subscript in the Laplacian $- \Delta_n$ for the remainder of this appendix. \lb{f6}} 
  
\begin{align}
-\Delta_{\bbS^1} &= -\df{\p^{2}}{\p\theta_{1}^{2}},     \no \\ 
- \Delta_{\bbS^2} &= - \f{1}{\sin(\theta_2)} 
\f{\partial}{\partial \theta_2} \bigg(\sin(\theta_2) \f{\partial}{\partial \theta_2}\bigg)  
- \f{1}{\sin^2(\theta_2)} \f{\partial^2}{\partial \theta_1^2},    \lb{A.15} \\ 
- \Delta_{\bbS^{n-1}} &= -\df{\p^{2}}{\p\theta_{n-1}^{2}} -(n-2)\cot(\theta_{n-1})\df{\p}{\p\theta_{n-1}} 
- [\sin(\theta_{n-1})]^{-2}\Delta_{\bbS^{n-2}}, \quad n \geq 3.    \no 
\end{align} 
Explicitly (cf.\ \cite[p.~19]{DX13}),  
\begin{align}
\begin{split}
-  \Delta_{\bbS^{n-1}} &= - [\sin(\theta_{n-1})]^{2-n} \f{\partial}{\partial \theta_{n-1}} 
\bigg[[\sin(\theta_{n-1})]^{n-2} \f{\partial}{\partial \theta_{n-1}}\bigg]      \\
& \quad - \sum_{j=1}^{n-2} \bigg(\prod_{k=j+1}^{n-1} [\sin(\theta_{k})]^{-2}\bigg) [\sin(\theta_j)]^{1-j}  
\f{\partial}{\partial \theta_j} \bigg[[\sin(\theta_j)]^{j-1} \f{\partial}{\partial \theta_j}\bigg]    \lb{A.16} \\
&= - \sum_{j=1}^{n-1} \bigg(\prod_{k=1}^{j-1} [\sin(\theta_{n-k})]^{-2}\bigg) [\sin(\theta_{n-j})]^{1-j-n}   \\
& \hspace*{1.4cm} \times 
\f{\partial}{\partial \theta_{n-j}} \bigg[[\sin(\theta_{n-j})]^{n-j-1} \f{\partial}{\partial \theta_{n-j}}\bigg]. 
\end{split}
\end{align}

\medskip 

\noindent {\bf Acknowledgments.}
We are indebted to Mark Ashbaugh, Andrei Mart\'inez-Finkel- shtein, and Gerald Teschl 
for very helpful comments. We are also very grateful to the anonymous referee for constructive criticism. 


\end{document}